\UseRawInputEncoding
\documentclass[10pt,reqno]{amsart}
\usepackage{indentfirst,latexsym,bm}
\usepackage{amsfonts}
\usepackage{amssymb}
\usepackage{amsmath}
\usepackage{amsbsy}
\usepackage{dsfont}
\usepackage{amsthm}
\usepackage{amscd}
\usepackage[all]{xy}
\usepackage{chemarrow}
\usepackage{color}
\usepackage{multirow}
\usepackage{hyperref}
\usepackage[titletoc]{appendix}
\usepackage{geometry}

\begin{document}
\title{On Hopf algebras over basic Hopf algebras of dimension 24}
\author{Rongchuan Xiong}
\address{School of Computer Science and Artificial Inteligent, Changzhou University, Changzhou 213164, China}
\email{rcxiong@foxmail.com}
\thanks{
The paper is partially supported by the NSFC (Grants No. 11926353, 11771142).
}
\maketitle

\newtheorem{question}{Question}
\newtheorem{defi}{Definition}[section]
\newtheorem{conj}{Conjecture}
\newtheorem{thm}[defi]{Theorem}
\newtheorem{lem}[defi]{Lemma}
\newtheorem{pro}[defi]{Proposition}
\newtheorem{cor}[defi]{Corollary}
\newtheorem{rmk}[defi]{Remark}
\newtheorem{Example}{Example}[section]

\theoremstyle{plain}
\newcounter{maint}
\renewcommand{\themaint}{\Alph{maint}}
\newtheorem{mainthm}[maint]{Theorem}

\theoremstyle{plain}

\newcommand{\C}{\mathcal{C}}
\newcommand{\D}{\mathcal{D}}
\newcommand{\A}{\mathcal{A}}
\newcommand{\cK}{\mathcal{K}}
\newcommand{\cH}{\mathcal{H}}
\newcommand{\De}{\Delta}
\newcommand{\M}{\mathcal{M}}
\newcommand{\K}{\mathds{k}}
\newcommand{\E}{\mathcal{E}}
\newcommand{\Pp}{\mathcal{P}}
\newcommand{\Lam}{\lambda}
\newcommand{\As}{^{\ast}}
\newcommand{\Aa}{a^{\ast}}
\newcommand{\B}{b^{\ast}}
\newcommand{\cF}{\mathcal{F}}
\newcommand{\HH}{\mathcal{K}}
\newcommand{\CYD}{{}^{\C}_{\C}\mathcal{YD}}
\newcommand{\AYD}{{}^{\A}_{\A}\mathcal{YD}}
\newcommand{\HYD}{{}^{H}_{H}\mathcal{YD}}
\newcommand{\ydH}{{}^{H}_{H}\mathcal{YD}}
\newcommand{\KYD}{{}^{K}_{K}\mathcal{YD}}
\newcommand{\DM}{{}_{D}\mathcal{M}}
\newcommand{\BN}{\mathcal{B}}
\newcommand{\NA}{\mathcal{B}}

\newcommand{\Ga}{g^{\ast}}
\newcommand{\X}{x^{\ast}}
\newcommand{\I}{\mathds{I}}
\newcommand{\Z}{\mathds{Z}}
\newcommand{\N}{\mathds{N}}
\newcommand{\G}{\mathcal{G}}
\newcommand{\roots }{\boldsymbol{\Delta }}

\newcommand\ad{\operatorname{ad}}
\newcommand{\Alg}{\Hom_{\text{alg}}}
\newcommand\Aut{\operatorname{Aut}}
\newcommand{\AuH}{\Aut_{\text{Hopf}}}
\newcommand\coker{\operatorname{coker}}
\newcommand\car{\operatorname{char}}
\newcommand\Der{\operatorname{Der}}
\newcommand\diag{\operatorname{diag}}
\newcommand\End{\operatorname{End}}
\newcommand\id{\operatorname{id}}
\newcommand\gr{\operatorname{gr}}
\newcommand\GK{\operatorname{GKdim}}
\newcommand{\Hom}{\operatorname{Hom}}
\newcommand\ord{\operatorname{ord}}
\newcommand\rk{\operatorname{rk}}
\newcommand\soc{\operatorname{soc}}

\newcommand{\bp}{\mathbf{p}}
\newcommand{\bq}{\mathbf{q}}
\newcommand\Sb{\mathbb S}
\newcommand\cR{\mathcal{R}}
\newcommand{\grAYD}{{}^{\gr\A}_{\gr\A}\mathcal{YD}}
\newcommand{\Dchaintwo}[3]{\xymatrix@C-4pt{\overset{#1}{\underset{x }{\circ}}\ar
@{-}[r]^{#2}
& \overset{#3}{\underset{v_1 }{\circ}}}}

\newcommand{\Dchaintwoa}[3]{\xymatrix@C-4pt{\overset{#1}{\underset{\  }{\circ}}\ar
@{-}[r]^{#2}
& \overset{#3}{\underset{\ }{\circ}}}}

\newcommand{\Dchainthree}[8]{\xymatrix@C-2pt{
\overset{#1}{\underset{#2}{\circ}}\ar  @ {-}[r]^{#3}  & \overset{#4}{\underset{#5
}{\circ}}\ar  @{-}[r]^{#6}
& \overset{#7}{\underset{#8}{\circ}} }}

\newcommand{\Dtriangle}[6]{
\xymatrix@R-12pt{  &    \overset{#1}{\underset{x}{\circ}} \ar  @{-}[dl]_{#2}\ar  @{-}[dr]^{#3} & \\
\overset{#4}{\underset{e_1}{\circ}} \ar  @{-}[rr]^{#5}  &  &\overset{#6}{\underset{v_1}{\circ}} }}

\begin{abstract}
We determine finite-dimensional Hopf algebras over an algebraically closed field of characteristic zero, whose Hopf coradical is isomorphic to a non-pointed basic Hopf algebra of dimension $24$ and the infinitesimal braidings are indecomposable objects. In particular, we obtain families of  new finite-dimensional Hopf algebras without the dual Chevalley property.
\end{abstract}
\section{Introduction}
Let $\K$ be an algebraically closed field of characteristic zero. It is a fundamental and difficult question  in Hopf algebra theory to classify finite-dimensional ones. The research in this direction is very rich. Most of the classification results consist of Hopf algebras that are basic or have the dual Chevalley property (that is, its coradical is a subalgebra) . But there are very few results on finite-dimensional Hopf algebras without the dual Chevalley property in the literature, unless examples without pointed duals were constructed in \cite{GG16,HX17,HX18b,MBGG,X17} via the generalized lifting method \cite{AC13}.

As a generalization of the \emph{lifting method}\cite{AS98},  the generalized lifting method gives a technical framework to classify the Hopf algebras without the dual Chevalley property. It consists of the following steps (see \cite{AC13}):
\begin{itemize}
  \item {\text{\rm Step\,1}.} Classify all Hopf algebras $L$ that are generated by a cosemisimple coalgebra.
  \item {\text{\rm Step\,2}.} Classify all connected graded
  Hopf algebras $R$ in  the category ${}_L^L\mathcal{YD}$ of left Yetter-Drinfeld modules over $L$.
  \item {\text{\rm Step\,3}.} Given $L$ and $R$ as in previous items, classify all Hopf algebras $A$ such that $\text{gr}\,A\cong R\sharp L$. Here $A$ is called a lifting of $R$ over $L$.
\end{itemize}

It works because of the following facts. Suppose that $A$ is a Hopf algebra over $\K$ and denote by $A_{[0]}$ the Hopf coradical of $A$ (it is generated by the coradical $A_0$ of $A$).  If $S_A(A_{[0]})\subseteq A_{[0]}$, then the standard filtration $\{A_{[n]}\}_{n\geq 0}$, defined recursively by $A_{[n]}=A_{[n-1]}\bigwedge A_{[0]}$, is a Hopf algebra filtration. Therefore the associated graded coalgebra
$\gr A=\oplus_{n=0}^{\infty}A_{[n]}/A_{[n-1]}$ with $A_{[-1]}=0$ is a Hopf algebra and so there is a connected graded braided Hopf algebra $R=\oplus_{n\geq 0}R(n)$ in ${}^{A_{[0]}}_{A_{[0]}}\mathcal{YD}$ such that $\gr A\cong R\sharp A_{[0]}$. Here  $R$ and $R(1)$ are called the \emph{diagram} and \emph{infinitesimal braiding} of $A$, respectively.

In this paper, following the work of \cite{GG16}, we fix a  Hopf algebra $\cK_{24,1}$ of dimension $24$ that is basic and generated by the coradical,  and  continue the study on Steps 2 and 3 in the lifting procedure.

Using the equivalence ${}_{\cK_{24,1}}^{\cK_{24,1}}\mathcal{YD}\cong {}_{\D(\cK_{24,1}^{cop})}\mathcal{M}$, we  determine simple Yetter-Drinfeld modules over $\cK_{24,1}$. It turns out there exist $144$ simple objects in ${}_{\cK_{24,1}}^{\cK_{24,1}}\mathcal{YD}$, among which there are $24$ one-dimensional objects
$\K_{\chi_{i,j,k}}$ with $(i,j,k)\in\I_{0,1}\times\I_{0,1}\times\I_{0,5}$ and $120$ two-dimensional objects $V_{i,j,k,\iota}$ with
$(i,j,k,\iota)\in\Lambda=\{(i,j,k,\iota)\mid i,j\in \I_{0,5},  k,\iota\in \I_{0,1}, j+3k\not\equiv 3(\iota+1)\mod 6\}$, see Theorem \ref{thmsimplemoduleD-24} for details.

Now we determine finite-dimensional Nichols algebras over simple objects in ${}_{\cK_{24,1}}^{\cK_{24,1}}\mathcal{YD}$ and the liftings of their bosonizations. We first discard Nichols algebras $\BN(V_{i,j,k,\iota})$  of infinite dimension.  Using the equivalence ${}_{\cK_{24,1}}^{\cK_{24,1}}\mathcal{YD}\cong {}_{\gr\A_{24,1}}^{\gr\A_{24,1}}\mathcal{YD}$, we transport the information of $\BN(V_{i,j,k,\iota})$ from the category ${}_{\cK_{24,1}}^{\cK_{24,1}}\mathcal{YD}$ to ${}_{\gr\A_{24,1}}^{\gr\A_{24,1}}\mathcal{YD}$, where $\A_{24,1}$ is the dual Hopf algebra of $\cK_{24,1}$; and we prove  $\BN(V_{i,j,k,\iota})\sharp\gr\A_{24,1}$ is infinite-dimensional by using the classification  results in \cite{H09}. Then we prove the remaining are finite-dimensional by computing their defining relations and PBW bases in ${}_{\cK_{24,1}}^{\cK_{24,1}}\mathcal{YD}$.
Finally,  we study the liftings of finite-dimensional Nichols algebras  following the techniques in \cite{AS98,GG16}. Consequently, we have the following theorem.
\begin{thm}\label{thmA-24}[Theorems \ref{thmFDNichols-24}~\&~\ref{thmFDHopfalgebra-24}]
Let $A$ be a finite-dimensional Hopf algebra over $\cK_{24,1}$ whose infinitesimal braiding $V$ is  indecomposable in ${}_{\cK_{24,1}}^{\cK_{24,1}}\mathcal{YD}$.   Then $V$ is isomorphic either to $\K_{\chi_{i,j,k}}$ for $(i,j,k)\in\Lambda^0$
or to $V_{i,j,k,\iota}$ for $(i,j,k,\iota)\in\cup_{i=1}^6\Lambda^i$, and $A$ is isomorphic  to one of the following objects:
\begin{itemize}
  \item $\bigwedge\K_{\chi_{i,j,k}}\sharp \cK_{24,1}$ for $(i,j,k)\in\Lambda^0$;
  \item $\BN(V_{i,j,k,\iota})\sharp \cK_{24,1}$ for $(i,j,k,\iota)\in\cup_{i=1}^6\Lambda^i-\Lambda^{1\ast}$;
  \item $\mathcal{C}_{i,j,k,\iota}(\mu)$ for $\mu\in\K$ and $(i,j,k,\iota)\in\Lambda^{1\ast}$.
\end{itemize}
\end{thm}

The Nichols algebra $\BN(V_{i,j,k,\iota})$ for $(i,j,k,\iota)\in\cup_{i=4}^6\Lambda^{i}$ is isomorphic as an algebra to a quantum plane. They have appeared in \cite{AGi17} and were shown that the braidings are of non-diagonal type. The Nichols algebra $\BN(V_{i,j,k,\iota})$ for $(i,j,k,\iota)\in\cup_{i=1}^3\Lambda^i$ is an algebra of dimension $18$ or $36$ with no quadratic relations. They are examples of Nichols algebra of non-diagonal type, which are  (up to isomorphism) arising from Nichols algebras of standard type $B_2$ by using the techniques in \cite{AA18}.

The Hopf algebras $\bigwedge\K_{\chi_{i,j,k}}\sharp \cK_{24,1}$ with $(i,j,k)\in\Lambda^0$ are the duals of pointed Hopf algebras of dimension $48$.
The Hopf algebras $\BN(V_{i,j,k,\iota})\sharp \cK_{24,1}$ for $(i,j,k,\iota)\in\cup_{i=4}^6\Lambda^i$ are the duals of pointed Hopf algebras of dimension $96$, $144$ or $288$. The Hopf algebras $\BN(V_{i,j,k,\iota})\sharp \cK_{24,1}$ for $(i,j,k,\iota)\in\Lambda^1$ or $\Lambda^2\cup\Lambda^3$ are the duals of pointed Hopf algebras of dimension $432$ or $864$, respectively. The Hopf algebras $\mathcal{C}_{i,j,k,\iota}(\mu)$ with $\mu\neq 0$ are non-trivial liftings of $\BN(V_{i,j,k,\iota})\sharp \cK_{24,1}$ for $(i,j,k,\iota)\in\Lambda^{1\ast}\subset\Lambda^1$. They constitute new examples of Hopf algebras without the dual Chevally property.

The paper is organized as follows: In section \ref{Preliminary}, we recall some basic knowledge and notations of Yetter-Drinfeld modules, Nichols algebras. In section \ref{secK-24}, we introduce the structure of the Hopf algebra $\cK_{24,1}$. In section \ref{secKnichols},   we determine all finite-dimensional Nichols algebras over  simple objects in ${}_{\cK_{24,1}}^{\cK_{24,1}}\mathcal{YD}$ and present them by generators and relations. In section \ref{secK-lifting}, we determine all finite-dimensional Hopf algebras over $\cK_{24,1}$, whose infinitesimal braidings are simple objects in ${}_{\cK_{24,1}}^{\cK_{24,1}}\mathcal{YD}$.

\section{Preliminaries}\label{Preliminary}
\subsection*{Conventions.} In the paper,  the base field $\K$ is algebraically closed of characteristic zero and $\xi$
is a primitive $6$th root of unity. Let $\Z_n:=\Z/n\Z$ and $\I_{k,n}:=\{k,k+1,\ldots,n\}$ for $n\geq k\geq 0$.

Let $H$ be a Hopf algebra over $\K$.  Denote by $\G(H)$ the set of group-like elements of $H$. For any $g,h\in\G(H)$, $\Pp_{g,h}(H)=\{x\in H\mid \Delta(x)=x\otimes g+h\otimes x\}$. Our references for Hopf algebra theory are \cite{M93,R11}.

\subsection{Yetter-Drinfeld modules and Hopf algebras with a projection}
Suppose that $H$ has bijective antipode and denote by ${}^{H}_{H}\mathcal{YD}$ the category of left Yetter-Drinfeld modules over $H$. Then ${}^{H}_{H}\mathcal{YD}$ is braided monoidal with the braiding $c_{V,W}$ for $V,W\in {}^{H}_{H}\mathcal{YD}$ given by
\begin{align}\label{equbraidingYDcat}
c_{V,W}:V\otimes W\mapsto W\otimes V,\ v\otimes w\mapsto v_{(-1)}\cdot w\otimes v_{(0)},\ \forall\,v\in V, w\in W.
\end{align}
In particular, $c:=c_{V,V}$ is a linear isomorphism satisfying the
braid equation $(c\otimes\text{id})(\text{id}\otimes c)(c\otimes\text{id})=(\text{id}\otimes c)(c\otimes\text{id})(\text{id}\otimes c)$, that is, $(V, c)$ is a braided vector space. Moreover, ${}^{H}_{H}\mathcal{YD}$ is rigid. The left dual $V\As$ is defined by
\begin{align*}
\langle h\cdot f,v\rangle=\langle f,S(h)v\rangle,\quad f_{(-1)}\langle f_{(0)},v\rangle=S^{-1}(v_{(-1)})\langle f, v_{(0)}\rangle.
\end{align*}

If $H$ is finite-dimensional,  then by \cite[Proposition\,2.2.1.]{AG99}, $\HYD\cong{}_{H^{\ast}}^{H^{\ast}}\mathcal{YD}$ as braided monoidal categories  via the functor $(F,\eta)$ defined as follows :
 $F(V)=V$ as a vector space,
\begin{align}
\begin{split}\label{eqVHD}
f\cdot v=f(S(v_{(-1)}))v_{(0)},\quad \delta(v)=\sum_{i}S^{-1}(h^i)\otimes h_i\cdot v,\  \text{and}\\
\eta:F(V)\otimes F(W)\mapsto F(V\otimes W), v\otimes w\mapsto w_{(-1)}\cdot v\otimes w_{(0)},
\end{split}
\end{align}
where $V,W\in\HYD$, $f\in H^{\ast},v\in V, w\in W$, $\{h_i\}$ and $\{h^i\}$ are the dual bases of $H$ and $H\As$.

For a Hopf algebra $R\in{}^{H}_{H}\mathcal{YD}$ that is braided, set $\Delta_R(r)=r^{(1)}\otimes r^{(2)}$ for the comultiplication. By the \emph{Radford biproduct or bosonization} of $R$ by $H$ (\cite{R11}), written as $R\sharp H$, means a usual Hopf algebra, as a vector space,
$R\sharp H=R\otimes H$, whose multiplication and comultiplication are provided by the smash product and smash coproduct, respectively:
\begin{align}
(r\sharp g)(s\sharp h)=r(g_{(1)}\cdot s)\sharp g_{(2)}h,\quad
\Delta(r\sharp g)=r^{(1)}\sharp (r^{(2)})_{(-1)}g_{(1)}\otimes (r^{(2)})_{(0)}\sharp g_{(2)}.\label{eqSmash}
\end{align}

\subsection{Nichols algebras and skew-derivations}
Let $H$ be a Hopf algebra with bijective antipode and $V\in{}^{H}_{H}\mathcal{YD}$. The \emph{Nichols algebra} $\BN(V)$ over $V$ is a $\N$-graded Hopf algebra $R=\oplus_{n\geq 0} R(n)$ in ${}^{H}_{H}\mathcal{YD}$ such that
\begin{align*}
R(0)=\mathds{k}, \quad R(1)=V,\quad
R\;\text{is generated as an algebra by}\;R(1),\quad
\Pp(R)=V.
\end{align*}

The Nichols algebra $\BN(V)$ is isomorphic to $T(V)/I(V)$, where $I(V)\subset T(V)$ is the largest $\mathds{N}$-graded ideal and coideal in $\HYD$ such that $I(V)\cap V=0$. Moreover, the ideal $I(V)$ is the kernel of the quantum symmetrizer associated to the braiding $c$ and  $\BN(V)$ as a coalgebra and an algebra depends only on $(V,c)$.

\begin{rmk}\label{rmkN-infity}
Suppose $W$ is a subspace of $(V, c)$ with $c(W\otimes W)\subset W\otimes W$, then $\dim \BN(W)=\infty$ means $\dim \BN(V)=\infty$. In particular, $\dim\BN(V)=\infty$ if the braiding $c$ has an eigenvector $v\otimes v\in V^{\otimes 2}$ with eigenvalue $1$ $($cf. \cite{G00}$)$.
\end{rmk}
\begin{rmk}
The Nichols algebra $\BN(V)$ is of \emph{diagonal type} if there is a linear basis $\{x_i,\ i\in\I_{1,n}\}$ such that $c(x_i\otimes x_j)=q_{ij}x_j\otimes x_i$ for some $q_{ij}\in\K$. The matrix $\mathbf{q} = (q_{ij})_{i,j\in \I_{1,n}}$ is called the matrix of the braiding. The \emph{generalized Dynkin diagram} of the matrix $\mathbf{q}$ is a graph with $n$ vertices, the vertex $i$ labeled with $q_{ii}$, and an arrow between the vertices $i$ and $j$ only if $q_{ij}q_{ji}\neq 1$, labelled with $q_{ij}q_{ji}$. Finite-dimensional Nichols algebras of diagonal type were classified by Heckenberger \cite{H09}, with the help of the Weyl groupoid and generalized root systems. Their defining relations were given by Angiono \cite{An13,An15}. See \cite{AA17} for a survey on Nichols algebras of diagonal type.
\end{rmk}

Let $C$ be a coalgebra, $D$ a subcoalgebra of $C$ and $W\in{}^C\mathcal{M}$. Denote the largest $D$-subcomodule of $W$ by
\begin{align*}
W(D)=\{w\in W\mid\delta(w)\in D\otimes W\}.
\end{align*}
\begin{pro}\cite[Proposition 8.8]{HS13}\label{pro-HS13-8.8}
Let $H$ be a Hopf algebra with bijective antipode, $N\in\HYD$ and $W\in{}_{\BN(N)\sharp H}^{\BN(N)\sharp H}\mathcal{YD}$.
Assume that $W$ is a semisimple object in the category of $\Z$-graded left Yetter-Drinfeld modules over $\BN(N)\sharp H$.
Let $\cK=\BN(W)$ in ${}_{\BN(N)\sharp H}^{\BN(N)\sharp H}\mathcal{YD}$, and define $M=W(H)$. Then there is a unique isomorphism
\begin{align*}
\cK\sharp\BN(N)\cong\BN(M\oplus N)
\end{align*}
of braided Hopf algebras in $\HYD$ which is the identity on $M\oplus N$.
\end{pro}

Now we recall the standard tool, the so called skew-derivation, for working with Nichols algebras. Let $(V,c)$ be a  $($rigid$)$ braided vector space of dimension $n$ and $\Delta^{i,m-i}:T^m(V)\rightarrow T^{i}(V)\otimes T^{m-i}(V)$ the $(i,m-i)$-homogeneous component of the comultiplication $\Delta:T(V)\rightarrow T(V)\otimes T(V)$ for $m\in\N$ and $k\in\I_{0,m}$. Given $f\in V^{\ast}$, the \emph{skew-derivation} $\partial_f\in\text{End}\,T(V)$ is given by
\begin{gather}\label{eqSkew-1}
\partial_f(v)=(f\otimes \id)\Delta^{1,m-1}(v):T^m(V)\rightarrow T^{m-1}(V),\quad v\in T^m(V),\ m\in\N.
\end{gather}
Let $\{v_i\}_{1\leq i\leq n}$ and $\{v^i\}_{1\leq i\leq n}$ be the dual bases of $V$ and $V\As$. We write $\partial_i:=\partial_{v^i}$ for simplicity.

This is useful for seeking the relations of $\BN(V)$ due to:
\begin{align}\label{Def-Nichols-IV}
I^m(V)=\{r\in T^m(V)\mid  \partial_{f_1}\partial_{f_2}\cdots\partial_{f_m}(r)=0,\ \forall\ f_i\in V\As\}.
\end{align}
Furthermore, $\partial_f$ can be defined on $\BN(V)$ and $\mathop\cap\limits_{f\in V\As}\ker\partial_f=\K$. For details, see \cite{AHS10,AA18}.

\section{The Hopf algebra $\cK_{24,1}$ and The category ${}_{\cK_{24,1}}^{\cK_{24,1}}\mathcal{YD}$}\label{secK-24}
We introduce the structures of the Hopf algebra $\cK_{24,1}$ and the category ${}_{\cK_{24,1}}^{\cK_{24,1}}\mathcal{YD}$ of Yetter-Drinfeld modules over $\cK_{24,1}$.

\subsection{The Hopf algebra $\cK_{24,1}$}
\begin{defi}\label{proStructureofH-24}
Let $\cK_{24,1}$ be the algebra generated by the elements $a$, $b$, $c$, $d$, subject to  the relations
  \begin{gather}
  \begin{split}\label{defiH-1}
  a^{6}=1,\quad b^2=0,\quad c^2=0, \quad d^{6}=1,\quad a^2=d^2,\quad ad=da,\quad bc=0=cb,\\
  ab=\xi ba,\quad ac=\xi ca,\quad db=-\xi bd,\quad dc=-\xi cd, \quad bd=ca,\quad ba=cd.
  \end{split}
  \end{gather}
\end{defi}
$\cK_{24,1}$  admits a Hopf algebra structure, where the coalgebra structure and antipode are given as follows:
  \begin{align}
  \begin{split}\label{defiH-2-24}
  \Delta(a)=a\otimes a+b\otimes c,\quad \Delta(b)=a\otimes b+b\otimes d,\\
  \Delta(c)=c\otimes a+d\otimes c,\quad \Delta(d)=d\otimes d+c\otimes b,\\
  \epsilon(a)=1,\quad \epsilon(b)=0,\quad \epsilon(c)=0,\quad \epsilon(d)=1,\\
  \end{split}
  \end{align}
  \begin{align*}
  S(a)=a^{-1},\quad S(b)=-\xi^{-1}ca^{-2}, \quad S(c)=\xi^{-1}ba^{-2},\quad S(d)=d^{-1}=da^{-2}.
  \end{align*}

\begin{rmk}\label{rmkHdualtoA-24}
\begin{enumerate}
\item The set $ \{a^i,da^i,ba^i,ca^i,\ i\in\I_{0,5}\}$ is a linear basis of $\cK_{24,1}$.

\item $\G(\cK_{24,1})=\K\{1,a^3,da^{2},da^{5}\}$, $\Pp_{1,da^5}(\cK_{24,1})=\K\{1-da^5,ca^5\}$ and $\Pp_{1,g}(\cK_{24,1})=\K\{1-g\}$ for $g\in\K\{a^3,da^2\}$.

\item Let $\{(a^i)\As,(ba^i)\As,(ca^i)\As,(da^i)\As,\;i\in\I_{0,5}\}$ be the dual basis of $\cK_{24,1}$ and set
\begin{gather*}
\widetilde{x}=\sum_{i=0}^{5} (ba^i)^{\ast}+ (ca^i)^{\ast},\quad
\widetilde{g}=\sum_{i=0}^{5} \xi^{i}(a^i)^{\ast}+ \xi^{i+1}(da^i)^{\ast},\quad
\widetilde{h}=\sum_{i=0}^{5} (a^i)^{\ast}-(da^i)^{\ast}.
\end{gather*}
The multiplication of $\cK_{24,1}$ implies that
\begin{gather*}
\widetilde{g}^{6}=1,\quad \widetilde{h}^2=1,\quad \widetilde{h}\widetilde{g}=\widetilde{g}\widetilde{h},\quad \widetilde{g}\widetilde{x}=\widetilde{x}\widetilde{g},\quad
\widetilde{h}\widetilde{x}=-\widetilde{x}\widetilde{h},\\
\Delta(\widetilde{x})=\widetilde{x}\otimes \epsilon+\widetilde{g}\widetilde{h}\otimes \widetilde{x},\quad
\Delta(\widetilde{g})=\widetilde{g}\otimes \widetilde{g},\quad
\Delta(\widetilde{h})=\widetilde{h}\otimes \widetilde{h}.
\end{gather*}
In particular, $\G(\cK_{24,1}\As)\cong \Z_6\otimes \Z_2$ with generators $\widetilde{g}$ and $\widetilde{h}$.
\end{enumerate}
\end{rmk}

Let $\A_{24,1}$ be the Hopf algebra generated by elements $g$, $h$, $x$ subject to the relations
\begin{gather*}
g^{6}=1,\quad h^2=1,\quad gh=hg,\quad gx=xg,\quad hx=-xh,\quad x^2=1-g^2,
\end{gather*}
where $g,h\in\G(\A_{24,1})$ and $x\in\Pp_{1,gh}(\A_{24,1})$. It is a pointed Hopf algebra of dimension $24$ appeared in \cite{G00b} (see also \cite{KR06}). We now  build the Hopf algebra isomorphism $\A_{24,1}\cong \cK_{24,1}^{\ast}$.

\begin{lem}\label{lemAtoHdual-24}
Let $\psi:\A_{24,1}\mapsto \cK_{24,1}^{\ast}$ be the algebra map given by
\begin{align*}
\psi(g)=\widetilde{g},\quad
\psi(h)=\widetilde{h},\quad
\psi(x)=\sqrt{1-\xi^2}\widetilde{x}.
\end{align*}
Then $\psi$ is a Hopf algebra isomorphism.
\end{lem}
\begin{proof}
By Remark \ref{rmkHdualtoA-24} $(3)$, $\psi$ is a bialgebra morphism and $\psi(\A_{24,1})$ contains properly $\G(\cK_{24,1}\As)$. Since $\dim\G(\cK_{24,1}\As)=12$,  by the Nichols-Zoeller Theorem,  $\psi$ must be  epimorphic. The lemma follows since $\dim \A_{24,1}=\dim \cK_{24,1}\As=24$.
\end{proof}

\begin{rmk}\label{rmkAtoHdual-24}
\begin{enumerate}
\item The set $\{ g^j, g^jh, g^jx, g^jhx,\ j\in\I_{0,5}\}$ is a linear basis of $\A_{24,1}$. We have
\begin{align*}
\psi(g^j)&=\sum_{i=0}^{5} \xi^{ij}(a^i)^{\ast}+ \xi^{ij+j}(da^i)^{\ast},\quad
\psi(g^jh)=\sum_{i=0}^{5} \xi^{ij}(a^i)^{\ast}-\xi^{ij+j}(da^i)^{\ast},\\
\psi(g^jx)&=\sqrt{1-\xi^2}\sum_{i=0}^{5} \xi^{ij+j}(ba^i)^{\ast}+\xi^{ij+j}(ca^i)^{\ast},\\
\psi(g^jhx)&=\sqrt{1-\xi^2}\sum_{i=0}^{5} \xi^{ij+j}(ba^i)^{\ast}-\xi^{ij+j} (ca^i)^{\ast}.
\end{align*}
\item It is clear that $\gr\A_{24,1}=\BN(X)\sharp \K[\Gamma]$, where $\Gamma\cong  \Z_{6}\times  \Z_2$ with generators $g,h$ and $X:=\K\{x\}\in{}_{\Gamma}^{\Gamma}\mathcal{YD}$  with $g\cdot x=x,h\cdot x=-x$ and $\delta(x)=gh\otimes x$.
Furthermore,  by \cite[Proposition\,4.2]{GM10}, $\gr\A_{24,1}\cong (\A_{24,1})_{\sigma}$, where $\sigma$ is a  Hopf 2-cocycle given by
\begin{align}\label{eqSigma}
\sigma=\epsilon\otimes\epsilon- \zeta, \text{where }  \zeta(x^ig^jh^k,x^mg^nh^l)=(-1)^{mk}\delta_{2,i+m},
\end{align}
for  $ i,k,m,l\in\I_{0,1}, j,n\in\I_{0,5}$.
Then by \cite[Theorem\, 2.7]{MO99}, ${}_{\A_{24,1}}^{\A_{24,1}}\mathcal{YD}\cong{}_{\gr\A_{24,1}}^{\gr\A_{24,1}}\mathcal{YD}$ as braided monoidal categories via the tensor functor $(G, \gamma)$ defined as follows: $G(V)=V$ as vector spaces and coactions,  transforming the action $\cdot$ to
\begin{align}
\begin{split}\label{formulaecocycle}
  k\cdot_{\sigma}v=\sigma(k_{(1)},v_{(-2)})\sigma^{-1}(k_{(2)}v_{(-1)}S(k_{(4)}),k_{(5)})k_{(3)}\cdot v_{(0)},\quad k\in \gr\A_{24,1}, \\
  \gamma:G(V)\otimes G(W)\longrightarrow G(V\otimes W),\quad v\otimes w\mapsto\sigma(v_{(-1)},w_{(-1)})v_{(0)}\otimes w_{(0)}.
\end{split}
\end{align}
\end{enumerate}
\end{rmk}

Now we give explicitly the structure of the Drinfeld double $\D(\cK_{24,1}^{cop})$ of $\cK_{24,1}^{cop}$.
From now on, we set $\D:=\D(\cK_{24,1}^{cop}):=\A_{24,1}^{cop\,op}\otimes \cK_{24,1}^{cop}$ for convenience. Recall that $\D(H^{cop})=H^{\ast\, op\, cop}\otimes H^{cop}$ is a Hopf algebra with the tensor product coalgebra structure and the algebra structure given by
$
(p\otimes a)(q\otimes b)=p\langle q_{(3)}, a_{(1)}\rangle q_{(2)}\otimes a_{(2)}\langle q_{(1)}, S^{-1}(a_{(3)})\rangle b
$.
\begin{pro}
$\D=\A_{24,1}^{cop\,op}\otimes \cK_{24,1}^{cop}$ is isomorphic to the algebra generated  by the elements $g$, $h$, $x$, $a$, $b$, $c$, $d$, subject to  the relations in $\cK_{24,1}^{cop}$, the relations in $\A_{24,1}^{cop\,op}$ and
\begin{gather*}
ag=ga,\quad ah=ha,\quad dg=gd,\quad dh=hd,\\
bg=gb,\quad bh=-hb,\quad cg=gc,\quad ch=-hc,\\
ax-\xi^{-1} xa=-\sqrt{1-\xi^2}\xi^{-1}(c-ghb),\quad dx+\xi^{-1} xd=-\sqrt{1-\xi^2}\xi^{-1}(ghc-b),\\
bx-\xi^{-1} xb=-\sqrt{1-\xi^2}\xi^{-1}(d-gha),\quad cx+\xi^{-1} xc=-\sqrt{1-\xi^2}\xi^{-1}(ghd-a).
\end{gather*}
\end{pro}

\subsection{The representation of $\D(\cK_{24,1}^{cop})$}
 We begin this subsection by describing one-dimensional objects in ${}_{\D}\mathcal{M}$.
\begin{lem}\label{lemOnesimpleD-24}
For $(i,j,k)\in\I_{0,1}\times\I_{0,1}\times\I_{0,5}$, there is a one-dimensional object $\K_{\chi_{i,j,k}}\in{}_{\D}\mathcal{M}$, where $\chi_{i,j,k}$ is  given by
\begin{gather*}
\chi_{i,j,k}(g)=(-1)^i,\quad\chi_{i,j,k}(h)=(-1)^j,\quad \chi_{i,j,k}(x)=0,\\
\chi_{i,j,k}(a)=\xi^k,\quad \chi_{i,j,k}(b)=0,\quad\chi_{i,j,k}(c)=0,\quad\chi_{i,j,k}(d)=(-1)^i(-1)^j\xi^k.
\end{gather*}
Any one-dimensional object in ${}_{\D}\mathcal{M}$ is isomorphic to $\K_{\chi_{i,j,k}}$ for some $(i,j,k)\in\I_{0,1}\times\I_{0,1}\times\I_{0,5}$.
\end{lem}
\begin{proof}
Let $\chi\in \G(\D^{\ast})$. The relations $a^{6}=1=g^{6}$, $d^{6}=1=h^2$, $b^2=0=c^2$ imply that $\chi(a)^{6}=1=\chi(g)^{6}$, $\chi(d)^{6}=1=\chi(h)^{2}$, $\chi(b)=0=\chi(c)$. Then the relations $hx=-xh$ and $x^2=1-g^2$ yield $\chi(x)=0$ and $\chi(g)^2=1$. From the relation $bx-\xi^{-1} xb=-\sqrt{1-\xi^2}\xi^{-1}(d-gha)$, it follows that $\chi(d)=\chi(g)\chi(g)\chi(a)$. Therefore, $\chi$ is completely determined by $\chi(a)$, $\chi(g)$ and $\chi(h)$ and then $\chi=\chi_{i,j,k}$ for some $(i,j,k)\in\I_{0,1}\times\I_{0,1}\times\I_{0,5}$. Consequently, any one-dimensional $\D$-module is isomorphic to $\K_{\chi_{i,j,k}}$ for some $(i,j,k)\in\I_{0,1}\times\I_{0,1}\times\I_{0,5}$. It is clear that these modules are pairwise non-isomorphic in ${}_{\D}\mathcal{M}$.
\end{proof}

Next, we describe two-dimensional simple objects in ${}_{\D}\mathcal{M}$. Let
\begin{align*}
\Lambda:=\{(i,j,k,\iota)\mid i,j\in \I_{0,5},\;  k,\iota\in \I_{0,1},\; j+3k\neq 3(\iota+1)\mod 6\}.
\end{align*}
Clearly, $|\Lambda|=120$.
\begin{lem}\label{lemTwosimpleD-24}
For any $(i,j,k,\iota)\in\Lambda$, there is a $2$-dimensional simple object  $V_{i,j,k,\iota}\in{}_{\D}\mathcal{M}$, whose matrices defining the $\D$-action on a fixed basis are given by
\begin{align*}
    [a]&=\left(\begin{array}{ccc}
                                   (-1)^{\iota+1}\xi^i & 0\\
                                    0  &  (-1)^{\iota+1}\xi^{i-1}
                                 \end{array}\right),\
    [d]=\left(\begin{array}{ccc}
                                   \xi^i & 0\\
                                    0  & -\xi^{i-1}
                                 \end{array}\right),\
    [b]=\left(\begin{array}{ccc}
                                   0 & (-1)^{\iota}\\
                                   0 & 0
                                 \end{array}\right),\\
    [c]&=\left(\begin{array}{ccc}
                                   0 & 1\\
                                   0 & 0
                                 \end{array}\right),\
    [g]=\left(\begin{array}{ccc}
                                   \xi^j & 0\\
                                   0   & \xi^j
                                 \end{array}\right),\
    [h]=\left(\begin{array}{ccc}
                                   (-1)^k & 0\\
                                   0   & (-1)^{k+1}
                                 \end{array}\right),\\
    [x]&=\left(\begin{array}{ccc}
                                   0 & (1-\xi^2)^{-\frac{1}{2}}\xi^{1-i}(\xi^j(-1)^k-(-1)^{\iota})\\
                                   -\sqrt{1-\xi^2}\xi^{i-1}(\xi^j(-1)^k+(-1)^{\iota})& 0
                                 \end{array}\right).
\end{align*}
Any two-dimensional simple object in ${}_{\D}\mathcal{M}$ is isomorphic to $V_{i,j,k,\iota}$ for  some $(i,j,k,\iota)\in\Lambda$. Furthermore, $V_{i,j,k,\iota}\cong V_{p,q,r,\kappa}$ if and only if $(i,j,k,\iota)=(p,q,r,\kappa)$.
\end{lem}
\begin{proof}
Let $V$ be a simple $\D$-module of dimension $2$. As the generators $g$, $h$, $a$, $d$ commute with each other and $g^{6}=h^2=a^{6}=d^{6}=1$, we may assume that the matrix defining the action on $V$ are of the form
\begin{align*}
    [g]&=\left(\begin{array}{ccc}
                                   g_1 & 0\\
                                   0   & g_2
                                 \end{array}\right),\,
    [h]=\left(\begin{array}{ccc}
                                   h_1 & 0\\
                                   0   & h_2
                                 \end{array}\right),\,
    [x]=\left(\begin{array}{ccc}
                                   x_1 & x_2\\
                                   x_3 & x_4
                                 \end{array}\right),\,
    [a]=\left(\begin{array}{ccc}
                                   a_1 & 0\\
                                    0  & a_2
                                 \end{array}\right),\\
    [d]&=\left(\begin{array}{ccc}
                                   d_1 & 0\\
                                    0  & d_2
                                 \end{array}\right),\quad
    [b]=\left(\begin{array}{ccc}
                                   b_1 & b_2\\
                                   b_3 & b_4
                                 \end{array}\right),\quad
    [c]=\left(\begin{array}{ccc}
                                   c_1 & c_2\\
                                   c_3 & c_4
                                 \end{array}\right),
\end{align*}
where $a_1^{6}=1=a_2^{6}$, $d_1^{6}=1=d_2^{6}$, $g_1^{6}=1=g_2^{6}$, $h_1^2=1=h_2^2$. The relations $xh=-hx$, $bh=-hb$ and $ch=-hc$ imply that
\begin{align*}
x_1&=0=x_4, \quad(h_1+h_2)x_2=0=(h_1+h_2)x_3, \\
b_1&=0=b_4, \quad(h_1+h_2)b_2=0=(h_1+h_2)b_3, \\
c_1&=0=c_4, \quad(h_1+h_2)c_2=0=(h_1+h_2)c_3.
\end{align*}
We claim that $h_1=-h_2$. Indeed, if $h_1+h_2\neq 0$, then   $x_2=0=x_3$, $b_2=0=b_3$, $c_2=0=c_3$, which implies that $[b]$, $[c]$, $[x]$ are zero matrices  and hence $V$ is not a simple $\D$-module.

Now we claim that $g_1=g_2$. Indeed, if $g_1\neq g_2$, then the relations $gx=xg$, $bg=gb$ and $cg=gc$ yield
$(g_1-g_2)x_2=0=(g_1-g_2)x_3,$ $(g_1-g_2)b_2=0=(g_1-g_2)b_3$, $(g_1-g_2)c_2=0=(g_1-g_2)c_3$,
 which implies that  $[b]$, $[c]$, $[x]$ are zero matrices and hence $V$ is not simple.

From the relations $b^2=0=c^2$ and $bc=0=cb$, we have that
\begin{align*}
b_2b_3=0=c_2c_3,\quad b_2c_3=0=b_3c_2,\quad c_2b_3=0=c_3b_2.
\end{align*}
By permuting the elements of the basis, we may assume that $b_3=0=c_3$.
The relations $ax-\xi^{-1} xa=-\sqrt{1-\xi^2}\xi^{-1}(c-ghb)$ and $dx+\xi^{-1} xd=-\sqrt{1-\xi^2}\xi^{-1}(ghc-b)$ imply that
\begin{align}
\begin{split}\label{eq1}
a_1x_2-\xi^{-1}a_2x_2&=-\sqrt{1-\xi^2}\xi^{-1}(c_2-g_1h_1b_2),\\
a_2x_3-\xi^{-1} a_1x_3&=-\sqrt{1-\xi^2}\xi^{-1}(c_3-g_2h_2b_3),
\end{split}\\
\begin{split}\label{eq2}
d_1x_2+\xi^{-1} d_2x_2&=-\sqrt{1-\xi^2}\xi^{-1}(g_1h_1c_2-b_2),\  d_2x_3+\xi^{-1} \\ d_1x_3&=-\sqrt{1-\xi^2}\xi^{-1}(g_2h_2c_3-b_3).
\end{split}
\end{align}

We claim that $b_2\neq 0$ or $c_2\neq 0$. Suppose that  $b_2=0=c_2$. Then  $x_2x_3\neq 0$ and  by equations $\eqref{eq1}$, $\eqref{eq2}$, we have
\begin{align*}
a_1x_2-\xi^{-1} a_2x_2=0,\, a_2x_3-\xi^{-1} a_1x_3=0,\, d_1x_2+\xi^{-1} d_2x_2=0,\, d_2x_3+\xi^{-1} d_1x_3=0.
\end{align*}
 Hence $a_1-\xi^{-1} a_2=0$ and $a_2-\xi^{-1} a_1=0$, which implies that $a_1=0=a_2$, a contradiction. Thus the claim follows. We may also assume that \textbf{$c_2=1$}.

The relations $ab=\xi ba$, $ac=\xi ca$, $db=-\xi bd$ and $dc=-\xi cd$ imply  $a_1=\xi a_2$, $d_1=-\xi d_2$.
The relations $bd=ca$ and $ba=cd$ yield $b_2^2=1$ and $a_2=b_2d_2$. From the relations $bx-\xi^{-1} xb=-\sqrt{1-\xi^2}\xi^{-1}(d-gha)$ and $cx+\xi^{-1} xc=-\sqrt{1-\xi^2}\xi^{-1}(ghd-a)$, it follows that
\begin{align*}
b_2x_3&=-\sqrt{1-\xi^2}\xi^{-1}(d_1-g_1h_1a_1),\, b_2x_3=\sqrt{1-\xi^2}(d_2-g_2h_2a_2),\\
x_3&=-\sqrt{1-\xi^2}\xi^{-1}(g_1h_1d_1-a_1),\, x_3=-\sqrt{1-\xi^2}(g_2h_2d_2-a_2),
\end{align*}
which implies that $x_3=-\sqrt{1-\xi^2}\xi^{-1}(b_2+g_1h_1)d_1$. By equations $\eqref{eq1}$ and $\eqref{eq2}$,  $x_2=(1-\xi^2)^{-\frac{1}{2}}\xi d_1^{-1}(g_1h_1-b_2)$.

The relations $x^2=1-g^2$ and $a^2=d^2$ imply $x_2x_3=1-g_1^2$ and $a_1^2-d_1^2=0=a_2^2-d_2^2$. Indeed, since $a_2=b_2d_2$, $a_1=\xi a_2$ and $d_2=-\xi^{-1} d_1$, it follows that $a_1=-b_2d_1$ and hence $a_1^2-d_1^2=0=a_2^2-d_2^2$. Similarly, the relation $x_2x_3=1-g_1^2$ holds.

Consequently,  we have
\begin{align*}
    [a]&=\left(\begin{array}{ccc}
                                   -\Lam_4\Lam_1 & 0\\
                                    0  &  -\xi^{-1} \Lam_4\Lam_1
                                 \end{array}\right),\,
    [d]=\left(\begin{array}{ccc}
                                   \Lam_1 & 0\\
                                    0  & -\xi^{-1} \Lam_1
                                 \end{array}\right),\,
    [b]=\left(\begin{array}{ccc}
                                   0 & \Lam_4\\
                                   0 & 0
                                 \end{array}\right),\\
    [c]&=\left(\begin{array}{ccc}
                                   0 & 1\\
                                   0 & 0
                                 \end{array}\right),\,
    [g]=\left(\begin{array}{ccc}
                                   \Lam_2 & 0\\
                                   0   & \Lam_2
                                 \end{array}\right),\,
    [h]=\left(\begin{array}{ccc}
                                   \Lam_3 & 0\\
                                   0   & -\Lam_3
                                 \end{array}\right),\\
    [x]&=\left(\begin{array}{ccc}
                                   0 & (1-\xi^2)^{-\frac{1}{2}}\xi \Lam_1^{-1}(\Lam_2\Lam_3-\Lam_4)\\
                                   -\sqrt{1-\xi^2}\xi^{-1}\Lam_1(\Lam_2\Lam_3+\Lam_4)& 0
                                 \end{array}\right),
\end{align*}
where $\Lam_1^{6}=1$, $\Lam_2^{6}=1$, $\Lam_3^2=1$, $\Lam_4^2=1$ and $\Lam_2\Lam_3+\Lam_4\neq 1$. Set $\Lam_1=\xi^i$, $\Lam_2=\xi^j$, $\Lam_3=(-1)^k$ and $\Lam_4=(-1)^\iota$ for some $(i,j,k,\iota)\in\Lambda$. Then $V\cong V_{i,j,k,\iota}$.

Now we claim that $V_{i,j,k,\iota}\cong V_{p,q,r,\kappa}$ if and only if $(i,j,k,\iota)=(p,q,r,\kappa)$ in $\Lambda$.
Let $\Phi:V_{i,j,k,\iota}\mapsto V_{p,q,r,\kappa}$ be an isomorphism of $\D$-modules  and
$[\Phi]=(p_{i,j})_{i,j=1,2}$   the matrix of $\Phi$ in the given basis. Then $[c][\Psi]=[\Psi][c]$ and $[a][\Psi]=[\Psi][a]$, which implies that $p_{21}=0,\,p_{11}=p_{22}$ and $(\xi^p-\xi^i)p_{11}=0,\,(\xi^p-\xi^{i+1})p_{12}=0$. Consequently, $\xi^i=\xi^p$ and then  $p_{12}=0$ and $[\Phi]=p_{11}I$ where $I$ is the identity matrix. Similarly, we have $\xi^j=\xi^q$, $k=r$, $\iota=\kappa$.  The claim follows. \end{proof}

Finally, we describe all simple objects in ${}_{\D}\mathcal{M}$  up to isomorphism.
\begin{thm}\label{thmsimplemoduleD-24}
There exist $144$ simple  objects in ${}_{\D}\mathcal{M}$  up to isomorphism, among which $24$ one-dimensional objects are given in
Lemma \ref{lemOnesimpleD-24} and $120$ two-dimensional simple objects are given in Lemma \ref{lemTwosimpleD-24}.
\end{thm}
\begin{proof}
 By \cite[Proposition\,10.6.16.]{M93}, \cite[Proposition 2.2.1]{AG99} and Remark \ref{rmkAtoHdual-24},
 ${}_{\D}\mathcal{M}\cong {}^{\cK_{24,1}}_{\cK_{24,1}}\mathcal{YD}\cong {}^{\A_{24,1}}_{\A_{24,1}}\mathcal{YD}\cong {}^{\gr\A_{24,1}}_{\gr\A_{24,1}}\mathcal{YD}\cong{}_{\D(\gr\A_{24,1})}\mathcal{M}$. On the other hand, $\D(\gr\A)$ is isomorphic to  a lifting of a quantum plane, which is generated by the elements $g_1,g_2,g_3,g_4, x_1, x_2$, subject to the relations
\begin{gather*}
g_ig_j=g_jg_i,\quad  g_{1+k}^6=g_{2+k}^2=1,\quad x_k^2=0, \quad i,j\in\I_{0,5},k\in\I_{0,1},\\
x_1x_2+x_2x_1=g_1g_2g_4-1, \quad g_ix_1=\chi(g_i)x_1g_i,\quad g_ix_2=\chi^{-1}(g_i)x_2g_i,
\end{gather*}
where $\chi(g_1)=1,\chi(g_2)=\chi(g_4)=-1$, $\chi(g_3)=\xi$, $\Delta(g_i)=g_i\otimes g_i$, $\Delta(x_1)=x_1\otimes 1+ g_1g_2\otimes x_1$ and $\Delta(x_2)=x_2\otimes 1+g_4\otimes x_2$.
Then by \cite[Theorem 3.5]{AB04},  $\dim V<3$ for
 any simple $\D(\gr\A_{24,1})$-module $V$. Consequently, the proposition follows.
\end{proof}

\subsection{The category ${}_{\cK_{24,1}}^{\cK_{24,1}}\mathcal{YD}$}
Using the equivalence ${}_{\D(\cK_{24,1}^{cop})}\mathcal{M}\cong{}_{\cK_{24,1}}^{\cK_{24,1}}\mathcal{YD}$, we describe explicitly  simple objects in ${}_{\cK_{24,1}}^{\cK_{24,1}}\mathcal{YD}$. Using the equivalences ${}_{\cK_{24,1}}^{\cK_{24,1}}\mathcal{YD}\cong {}_{\A_{24,1}}^{\A_{24,1}}\mathcal{YD}\cong {}_{\gr\A_{24,1}}^{\gr\A_{24,1}}\mathcal{YD}$, we transport the information from the category ${}_{\cK_{24,1}}^{\cK_{24,1}}\mathcal{YD}$ to ${}_{\gr\A_{24,1}}^{\gr\A_{24,1}}\mathcal{YD}$.

\begin{lem}\label{lemoneSimpleobject-24}
Let $\K_{\chi_{i,j,k}}=\K\{v\}\in{}_{\D}\mathcal{M}$ for $(i,j,k)\in \I_{0,1}\times \I_{0,1}\times \I_{0,5}$. Then  $\K_{\chi_{i,j,k}}\in{}_{\cK_{24,1}}^{\cK_{24,1}}\mathcal{YD}$  with
\begin{align*}
a\cdot v&=\xi^k v,\quad b\cdot v=0, \quad c\cdot v=0,\quad d\cdot v=(-1)^{i+j}\xi^k v;\quad\delta(v)=d^ja^{3i-j}\otimes v.
\end{align*}
\begin{proof}
The $\cK_{24,1}$-action is given by the restriction of the character of $\D$. The coaction is of the form $\delta(v)=t\otimes v$, where $t\in G(\cK_{24,1})=\{1,a^3,da^{2},da^{-1}\}$ such that $\langle g, t\rangle v=(-1)^iv$ and $\langle h, t\rangle v=(-1)^jv$. Therefore, $\delta(v)=d^ja^{3i-j}\otimes v$.
\end{proof}
\end{lem}

\begin{cor}\label{cor-V-gr-1-dim}
Let $\K_{\chi_{i,j,k}}=\K\{v\}\in{}_{\cK_{24,1}}^{\cK_{24,1}}\mathcal{YD}$. Then $\K_{\chi_{i,j,k}}\in{}_{\gr\A_{24,1}}^{\gr\A_{24,1}}\mathcal{YD}$ with
\begin{align*}
g\cdot v=(-1)^iv,\quad h\cdot v=(-1)^jv,\quad x\cdot v=0,\quad \delta(v)=g^{-k}h^{i+j}\otimes v.
\end{align*}
\end{cor}
\begin{proof}
Since $\cK_{24,1}\cong\A_{24,1}\As$, by \cite[Proposition\,2.2.1.]{AG99}, we have the equivalence
${}_{\cK_{24,1}}^{\cK_{24,1}}\mathcal{YD}\cong {}_{\A_{24,1}}^{\A_{24,1}}\mathcal{YD}$ via the functor $(F,\eta)$ defined by \eqref{eqVHD}. More precisely,  by the formula \eqref{eqVHD},  Lemma \ref{lemoneSimpleobject-24} and Remark \ref{rmkAtoHdual-24} (1), we have $F(\K_{\chi_{i,j,k}})=\K_{\chi_{i,j,k}}\in{}_{\A_{24,1}}^{\A_{24,1}}\mathcal{YD}$ with
\begin{align*}
g\cdot v&=\langle g, S(d^ja^{3i-j})\rangle  v=\langle g, d^{j}a^{-j-3i}\rangle  v=(-1)^iv,\\
h\cdot v&=\langle h, d^{j}a^{-j-3i}\rangle  v=(-1)^jv,\quad
x\cdot v=\langle x, d^{j}a^{-j-3i}\rangle  v=0,\\
\delta(v)&=\sum_{i}S^{-1}((a^i)^\ast)\otimes a^i\cdot v+S^{-1}((da^i)^\ast)\otimes da^i\cdot v\\
&=S^{-1}(g^kh^{i+j})\otimes v=g^{-k}h^{i+j}\otimes v.
\end{align*}

Then by Remark \ref{rmkAtoHdual-24} (2), we have ${}_{\A_{24,1}}^{\A_{24,1}}\mathcal{YD}\cong {}_{\gr\A_{24,1}}^{\gr\A_{24,1}}\mathcal{YD}$ via the functor $(G,\gamma)$ defined by the formulae \eqref{eqSigma}--\eqref{formulaecocycle}, and then  $GF(\K_{\chi_{i,j,k}})=\K_{\chi_{i,j,k}}\in {}_{\gr\A_{24,1}}^{\gr\A_{24,1}}\mathcal{YD}$ with the module structure given by
\begin{align*}
h\cdot_{\sigma}v&=\sigma(h_{(1)},v_{(-2)})\sigma^{-1}(h_{(2)}v_{(-1)}S(h_{(4)}),h_{(5)})h_{(3)}\cdot v_{(0)}\\
&=\sigma(h,g^{-k}h^{i+j})\sigma^{-1}(g^{-k}h^{i+j},h)h\cdot v=h\cdot v=(-1)^jv,\\
g\cdot_{\sigma}v&=\sigma(g,g^{-k}h^{i+j})\sigma^{-1}(gg^{-k}h^{i+j}S(g),g)g\cdot v\\
&=\sigma(g,g^{-k}h^{i+j})\sigma^{-1}(g^{-k}h^{i+j},g)g\cdot v=g\cdot v=(-1)^iv,\\
x\cdot_{\sigma}v&=\sigma(x_{(1)},v_{(-2)})\sigma^{-1}(x_{(2)}v_{(-1)}S(x_{(4)}),x_{(5)})x_{(3)}\cdot v_{(0)}\\
&=\sigma(x,g^{-k}h^{i+j})\sigma^{-1}(g^{-k}h^{i+j},1)1\cdot v
+\sigma(gh,g^{-k}h^{i+j})\sigma^{-1}(xg^{-k}h^{i+j},1)1\cdot v\\
&\quad+\sigma(gh,g^{-k}h^{i+j})\sigma^{-1}(ghg^{-k}h^{i+j},1)x\cdot v\\
&\quad+\sigma(gh,g^{-k}h^{i+j})\sigma^{-1}(ghg^{-k}h^{i+j}S(x),1)gh\cdot v\\
&\quad+\sigma(gh,g^{-k}h^{i+j})\sigma^{-1}(ghg^{-k}h^{i+j}S(gh),x)gh\cdot v\\
&=\sigma(x,g^{-k}h^{i+j})\sigma^{-1}(g^{-k}h^{i+j},1) v
+\sigma(gh,g^{-k}h^{i+j})\sigma^{-1}(xg^{-k}h^{i+j},1) v\\
&\quad+\sigma(gh,g^{-k}h^{i+j})\sigma^{-1}(g^{1-k}h^{i+j+1},1)x\cdot v\\
&\quad+\sigma(gh,g^{-k}h^{i+j})\sigma^{-1}(g^{1-k}h^{i+j+1}S(x),1)gh\cdot v\\
&\quad+\sigma(gh,g^{-k}h^{i+j})\sigma^{-1}(g^{-k}h^{i+j},x)gh\cdot v\\
&=x\cdot v=0.
\end{align*}
\end{proof}

\begin{lem}\label{lemtwoSimpleobject-24}
  Let $V_{i,j,k,\iota}=\K\{v_1,v_2\}\in {}_{\D}\mathcal{M}$ for any $(i,j,k,\iota)\in\Lambda$. Then $V_{i,j,k,\iota}\in{}_{\cK_{24,1}}^{\cK_{24,1}}\mathcal{YD}$ with the module structure given by
\begin{align*}
a\cdot v_1&=(-1)^{\iota+1}\xi^iv_1,\quad b\cdot v_1=0,\quad c\cdot v_1=0,\quad d\cdot v_1=\xi^iv_1,\\
a\cdot v_2&=(-1)^{\iota+1}\xi^{i-1} v_2,\quad b\cdot v_2=(-1)^{\iota}v_1,\quad c\cdot v_2=v_1,\quad d\cdot v_2=-\xi^{i-1}v_2;
\end{align*}
and the comodule structure given by
\begin{enumerate}
  \item for $k=0$,
  \begin{align*}
  \delta(v_1)&=a^j\otimes v_1+(1-\xi^2)^{-\frac{1}{2}}x_2ba^{j-1}\otimes v_2,\\
  \delta(v_2)&=da^{j-1}\otimes v_2+(1-\xi^2)^{-\frac{1}{2}}x_1ca^{j-1}\otimes v_1;
  \end{align*}
  \item for $k=1$,
  \begin{align*}
  \delta(v_1)&=da^{j-1}\otimes v_1+(1-\xi^2)^{-\frac{1}{2}}x_2ca^{j-1}\otimes v_2,\\
  \delta(v_2)&=a^{j}\otimes v_2+(1-\xi^2)^{-\frac{1}{2}}x_1ba^{j-1}\otimes v_1;
  \end{align*}
\end{enumerate}
where $x_1=(1-\xi^2)^{-\frac{1}{2}}\xi^{1-i}(\xi^j(-1)^k-(-1)^{\iota})$ and $x_2=-\sqrt{1-\xi^2}\xi^{i-1}(\xi^j(-1)^k+(-1)^{\iota})$.
\end{lem}
\begin{proof}
Let $\{h_i\}_{i\in\I_{1,24}}$ and $\{h^i\}_{i\in\I_{1,24}}$ be the dual bases of $\cK_{24,1}$ and $\cK_{24,1}\As$. The $\cK_{24,1}$-action is given by the restriction of the $\D$-action and the $\cK_{24,1}$-comodule structure is given by $\delta(v)=\sum_{i=1}^{24}h_i\otimes h^i\cdot v$ for any $v\in V_{i,j,k}$.
By Lemma \ref{lemAtoHdual-24} and Remark \ref{rmkAtoHdual-24}, we have
\begin{align*}
(g^l)^\ast&=\frac{1}{12}\sum_{i=0}^{5} \xi^{-il}a^i+\xi^{-(i+1)l}da^i ,\quad
(g^lh)^\ast=\frac{1}{12} \sum_{i=0}^{5} \xi^{-il}a^i -\xi^{-(i+1)l} da^i ,\\
(g^lx)^\ast&=\frac{1}{12\sqrt{1-\xi^2}} \sum_{i=0}^{5} \xi^{-(i+1)l}ba^i  + \xi^{-(i+1)l} ca^i ,\\
(g^lhx)^\ast&=\frac{1}{12\sqrt{1-\xi^2}} \sum_{i=0}^{5} \xi^{-(i+1)l}ba^i -\xi^{-(i+1)l}ca^i .
\end{align*}
 Then the lemma follows by a direct computation.
 \end{proof}

\begin{rmk}\label{rmkDdual-H-24}
$V_{i,j,k,\iota}\As\cong V_{-i+4,-j,k+1,\iota+1}$ for all $(i,j,k,\iota)\in \Lambda$.
\end{rmk}
In this section, for $(i,j,k,\iota)\in\Lambda$, set
\begin{align*}
 x_1:&=(1-\xi^2)^{-\frac{1}{2}}\xi^{1-i}(\xi^j(-1)^k-(-1)^{\iota}), \\ x_2:&=-\sqrt{1-\xi^2}\xi^{i-1}(\xi^j(-1)^k+(-1)^{\iota}).
\end{align*}

\begin{cor}\label{rmkVA2-24}
Let $V_{i,j,k,\iota}=\K\{v_1,v_2\}\in{}_{\cK_{24,1}}^{\cK_{24,1}}\mathcal{YD}$ for $(i,j,k,\iota)\in\Lambda$. Then $V_{i,j,k,\iota}\in{}_{\gr\A_{24,1}}^{\gr\A_{24,1}}\mathcal{YD}$ with the module structure given by
\begin{align*}
g\cdot v_1&=\xi^{-j}v_1,\quad h\cdot v_1=(-1)^kv_1,\quad x\cdot v_1=(-1)^{k+1}x_2\xi^{-j}v_2,\\
g\cdot v_2&=\xi^{-j}v_2,\quad h\cdot v_2=(-1)^{k+1}v_2,\quad x\cdot v_2=0;
\end{align*}
and the comodule structure given as follows:
\begin{enumerate}
  \item for $\iota=0$,
  \begin{align*}
  \delta(v_1)=g^{-3-i}h\otimes v_1,\quad\delta(v_2)=g^{-2-i}\otimes v_2+\xi^{1-i}(1-\xi^2)^{-\frac{1}{2}}g^{-3-i}hx\otimes v_1;
  \end{align*}
  \item for $\iota=1$,
  \begin{align*}
  \delta(v_1)=g^{-i}\otimes v_1,\quad\delta(v_2)=g^{-i+1}h\otimes v_2-\xi^{1-i}(1-\xi^2)^{-\frac{1}{2}}g^{-i}x\otimes v_1.
  \end{align*}
\end{enumerate}
\end{cor}
\begin{proof}
Similar to the proof of Corollary \ref{cor-V-gr-1-dim}, using the equivalence ${}_{\cK_{24,1}}^{\cK_{24,1}}\mathcal{YD}\cong {}_{\A_{24,1}}^{\A_{24,1}}\mathcal{YD}$ via the functor $(F,\eta)$ defined by \eqref{eqVHD}, by  Lemma \ref{lemtwoSimpleobject-24}, we have
 $F(V_{i,j})\in{}_{\A_{24,1}}^{\A_{24,1}}\mathcal{YD}$ with the  comodule structure given in the corollary and the module structure given by
\begin{align*}
g\cdot v_1&=\xi^{-j}v_1,\quad h\cdot v_1=(-1)^kv_1,\quad x\cdot v_1=(-1)^{k+1}x_2\xi^{-j}v_2,\\
g\cdot v_2&=\xi^{-j}v_2,\quad h\cdot v_2=(-1)^{k+1}v_2,\quad x\cdot v_2=(-1)^{k+1}x_1\xi^{-j}v_1;
\end{align*}
Then using the equivalence $ {}_{\A_{24,1}}^{\A_{24,1}}\mathcal{YD}\cong {}_{\gr\A_{24,1}}^{\gr\A_{24,1}}\mathcal{YD}$ via the functor $(G,\gamma)$ defined by the formulae \eqref{eqSigma}--\eqref{formulaecocycle}, by a direct and tedious computation, we have $GF(V_{i,j,k,\iota})=V_{i,j,k,\iota}\in{}_{\gr\A_{24,1}}^{\gr\A_{24,1}}\mathcal{YD}$ with the structure given in the corollary.
\end{proof}

Finally, we describe the braiding of simple objects in ${}_{\cK_{24,1}}^{\cK_{24,1}}\mathcal{YD}$.
\begin{lem}\label{lembraidingone-24}
Let $\K_{\chi_{i,j,k}}=\K\{v\}\in{}_{\cK_{24,1}}^{\cK_{24,1}}\mathcal{YD}$ for $(i,j,k)\in \I_{0,1}\times\I_{0,1}\times \I_{0,5}$. Then the braiding of $\K_{\chi_{i,j,k}}$  is given by $c(v\otimes v)=(-)^{i(j+k)+j}v\otimes v$.
\end{lem}
\begin{proof}
By the formula \eqref{equbraidingYDcat} of the braiding in ${}_{\cK_{24,1}}^{\cK_{24,1}}\mathcal{YD}$  and Lemma \ref{lemoneSimpleobject-24}, we have
\begin{align*}
c(v\otimes v)=d^ja^{3i-j}\cdot v\otimes v=\xi^{(3i-j)k}(-1)^{(i+j)j}\xi^{jk}v\otimes v=(-1)^{ik+ij+j^2}v\otimes v.
\end{align*}
\end{proof}

For $(i,j,k,\iota)\in\Lambda$, we set
\begin{align*}
a_{12}:&=[(-1)^{(\iota+1)j}\xi^{ij}+(-1)^{(j-1)\iota}\xi^{(i+2)j}], \\ a_{11}:&=\frac{1}{1-\xi^2}(-1)^{\iota(j-1)}\xi^{(j-2)i+2j-1}[\xi^j-(-1)^{\iota}],\\
b_{12}:&=[(-1)^{(j-1)(\iota+1)}\xi^{ij}+(-1)^{j\iota}\xi^{(i+2)j}],\\
b_{11}:&=\frac{1}{1-\xi^2}(-1)^{j\iota}\xi^{2j+2+ij+4i}[\xi^j+(-1)^{\iota}].
\end{align*}
\begin{lem}\label{lembraidsimpletwo-24}
Let $V_{i,j,k,\iota}=\K\{v_1,v_2\}\in{}_{\cK_{24,1}}^{\cK_{24,1}}\mathcal{YD}$ for $(i,j,k,\iota)\in\Lambda$. Then the braiding of $V_{i,j,k,\iota}$  is given as follows:
\begin{enumerate}
  \item If $k=0$, then $c(\left[\begin{array}{ccc} v_1\\v_2\end{array}\right]\otimes\left[\begin{array}{ccc} v_1~v_2\end{array}\right])=$
  \begin{align*}
  \left[\begin{array}{ccc}
(-1)^{(\iota+1)j}\xi^{ij}v_1\otimes v_1&(-1)^{j\iota}\xi^{(i+2)j}v_2\otimes v_1+a_{12}v_1\otimes v_2\\
(-1)^{(\iota+1)(j-1)}\xi^{ij}v_1\otimes v_2 &(-1)^{(j-1)\iota}\xi^{(i+2)j}v_2\otimes v_2+a_{11}v_1\otimes v_1
  \end{array}\right].
  \end{align*}
  \item If $k=1$, then $ c(\left[\begin{array}{ccc} v_1\\v_2\end{array}\right]\otimes\left[\begin{array}{ccc} v_1~v_2\end{array}\right])=$
  \begin{align*}
  \left[\begin{array}{ccc}
(-1)^{(j-1)(\iota+1)}\xi^{ij}v_1\otimes v_1&(-1)^{(j-1)\iota}\xi^{(i+2)j}v_2\otimes v_1+b_{12}v_1\otimes v_2\\
(-1)^{(\iota+1)j}\xi^{ij}v_1\otimes v_2
  &(-1)^{j\iota}\xi^{(i+2)j}v_2\otimes v_2+b_{11}v_1\otimes v_1
         \end{array}\right].
  \end{align*}
\end{enumerate}
\end{lem}
\begin{proof}
If $k=0$, then by  Lemma \ref{lemtwoSimpleobject-24} and the formula \eqref{equbraidingYDcat}, we have
\begin{align*}
c(v_1\otimes v_1)&=a^j\cdot v_1\otimes v_1+(1-\xi^2)^{-\frac{1}{2}}x_2ba^{j-1}\cdot v_1\otimes v_2=(-1)^{(\iota+1)j}\xi^{ij}v_1\otimes v_1,\\
c(v_1\otimes v_2)&=(-1)^{j\iota}\xi^{(i+2)j}v_2\otimes v_1+a_{12}v_1\otimes v_2,\\
c(v_2\otimes v_1)&=da^{j-1}\cdot v_1\otimes v_2+(1-\xi^2)^{-\frac{1}{2}}x_1ca^{j-1}\cdot v_1\otimes v_1\\
&=(-1)^{(\iota+1)(j-1)}\xi^{ij}v_1\otimes v_2,\\
c(v_2\otimes v_2)&=(-1)^{(j-1)\iota}\xi^{(i+2)j}v_2\otimes v_2+a_{11}v_1\otimes v_1.
\end{align*}
Similarly, we can obtain the matrix of  the braiding associated with $V_{i,j,1,\iota}$.
\end{proof}

\section{Nichols algebras   in ${}_{\cK_{24,1}}^{\cK_{24,1}}\mathcal{YD}$}\label{secKnichols}
We determine all finite-dimensional Nichols algebras over simple objects in ${}_{\cK_{24,1}}^{\cK_{24,1}}\mathcal{YD}$. Let
\begin{gather*}
\Lambda^0:=\{(i,j,k)\in \I_{0,1}\times \I_{0,1}\times \I_{0,5}\mid  i(j+k)+j \equiv 1\mod 2 \}.
\end{gather*}
We shall show  that finite-dimensional Nichols algebras over one-dimensional objects in ${}_{\cK_{24,1}}^{\cK_{24,1}}\mathcal{YD}$ are parametrized by $\Lambda^0$.
\begin{pro}\label{proNicholsoneH-24}
Let $\K_{\chi_{i,j,k}}=\K\{v\}$ for $(i,j,k)\in \I_{0,1}\times \I_{0,1}\times \I_{0,5}$. Then
\begin{align*}
\BN(\K_{\chi_{i,j,k}})=\begin{cases}
\bigwedge \K_{\chi_{i,j,k}}, & (i,j,k)\in\Lambda^0;\\
\K[v], & \text{others}.
\end{cases}
\end{align*}
\end{pro}
\begin{proof}
It follows directly by Lemma \ref{lembraidingone-24} and Remark \ref{rmkN-infity}.
\end{proof}

Let $V_{i,j,k,\iota}=\K\{v_1,v_2\}\in{}_{\cK_{24,1}}^{\cK_{24,1}}\mathcal{YD}$ for $(i,j,k,\iota)\in\Lambda$. Then using the equivalence ${}_{\cK_{24,1}}^{\cK_{24,1}}\mathcal{YD}\cong{}_{\gr\A_{24,1}}^{\gr\A_{24,1}}\mathcal{YD}$,   we have $V_{i,j,k,\iota}\in{}_{\gr\A_{24,1}}^{\gr\A_{24,1}}\mathcal{YD}$, that is, $V_{i,j,k,\iota}\in{}_{\BN(X)\sharp \K[\Gamma]}^{\BN(X)\sharp \K[\Gamma]}\mathcal{YD}$ with the structure given by Corollary \ref{rmkVA2-24}.

By Proposition \ref{pro-HS13-8.8} (see also \cite[Theorem 1.1]{AA18}), $\BN(V_{i,j,k,\iota})\sharp\BN(X)\cong\BN(X\oplus X_{i,j,k,\iota})$ in ${}_{\Gamma}^{\Gamma}\mathcal{YD}$ which is the identity on $X\oplus X_{i,j,k,\iota}$, where $X_{i,j,k,\iota}=\K\{v_1\}\in{}_{\Gamma}^{\Gamma}\mathcal{YD}$ with
\begin{gather*}
g\cdot v_1=\xi^{-j}v_1,\quad h\cdot v_1=(-1)^kv_1,\quad \delta(v_1)=g^{-3(\iota-1)-i}h^{\iota-1}\otimes v_1.
\end{gather*}
It is clear that $\BN(X\oplus X_{i,j,k,\iota})$ is of diagonal type with the generalized Dynkin diagram  given by
\xymatrix@C+45pt{\overset{-1 }{\underset{x  }{\circ}}\ar
@{-}[r]^{(-1)^{k+\iota-1}\xi^{-j}\quad\quad}
& \overset{(-)^{(k+j)(\ell-1)}\xi^{ij}}{\underset{v_1 }{\circ}}}.

Now we show that infinite-dimensional Nichols algebras over two-dimensional simple objects in ${}_{\cK_{24,1}}^{\cK_{24,1}}\mathcal{YD}$ are parametrized by the following subsets:
\begin{align*}
\Lambda^{0\ast}:&=\{(i,j,k,\iota)\in\Lambda\mid 3(k+j)(\iota-1)+ij \text{ or }\\&\quad\quad 3(k+\iota)+3(k+j)(\iota-1)+(i-1)j\equiv 0\mod 6\};\\
\Lambda^{0\ast\ast}:&=\{(i,j,k,\iota)\in\Lambda\mid i=5, j\in\{1,5\},  k+\iota+1\equiv 0\mod 2\}.
\end{align*}
and present finite-dimensional ones by generators and relations.
\begin{lem}\label{lemH-infity-24}
Let $(i,j,k,\iota)\in\Lambda^{0\ast}\cup\Lambda^{0\ast\ast}$. Then $\dim\BN(V_{i,j,k,\iota})=\infty$.
\end{lem}
\begin{proof}
It suffices to show that $\dim\BN(X\oplus X_{i,j,k,\iota})=\infty$.
If $(i,j,k,\iota)\in\Lambda^{0\ast}$, then the Dynkin diagram is $\xymatrix@C+45pt{\overset{-1 }{\underset{x  }{\circ}}\ar
@{-}[r]^{(-1)^{k+\iota-1}\xi^{-j}\quad\quad}
& \overset{(-)^{(k+j)(\iota-1)}\xi^{ij}}{\underset{v_1 }{\circ}}}$.

If $(i,j,k,\iota)\in\Lambda^{0\ast\ast}$, then the Dynkin diagram of $X\oplus X_{i,j,k,\iota}$ is given by
\xymatrix@C+15pt{\overset{-1 }{\underset{x  }{\circ}}\ar
@{-}[r]^{\xi^{-j}}
& \overset{\xi^{-j}}{\underset{v_1 }{\circ}}}. These diagrams do not appear in \cite[Table 1]{H09}, that is, they have infinite root systems. Therefore,  $\dim\BN(X\oplus X_{i,j,k,\iota})=\infty$.
\end{proof}

\begin{pro}\label{proV1-24}
Let $\Lambda^1=\{(i,j,k,\iota)\in\Lambda\mid 3(k+\iota-1)-j\equiv\pm 1\mod 6, 3(k+j)(\iota-1)+ij\equiv 3\pm 1\mod 6\}$. The Nichols algebra $\BN(V_{i,j,k,\iota})$ for $(i,j,k,\iota)\in\Lambda^1$ is generated by $v_1,v_2$, subject to the relations
   \begin{gather}
   v_1^3=0,\quad\xi^{2j}v_1^2v_2+(-1)^{\iota}\xi^{-2j}v_1v_2v_1+v_2v_1^2=0, \label{eqR1-1}\\
   v_1v_2^2+(-1)^{\iota}v_2v_1v_2+v_2^2v_1=0,\label{eqR1-2}\\
    \frac{(-1)^{\iota}(1-\xi^{2j})\xi^4}{1+\xi^5}v_1^2v_2+\frac{(1-\xi^{-2j})\xi^4}{1+\xi^5}v_1v_2v_1+v_2^3=0.\label{eqR1-3}
   \end{gather}
\end{pro}
\begin{proof}
By Lemma \ref{lembraidsimpletwo-24}, the braiding of $V_{i,j,k,\iota}$ is given by
\begin{align*}
  c(\left[\begin{array}{ccc} v_1\\v_2\end{array}\right]\otimes\left[\begin{array}{ccc} v_1~v_2\end{array}\right])=
  \left[\begin{array}{ccc}
\xi^{2j} v_1\otimes v_1& (-1)^{\iota}\xi^{-2j} v_2\otimes v_1- v_1\otimes v_2\\
(-1)^{\iota+1}\xi^{2j} v_1\otimes v_2 & \xi^{-2j}v_2\otimes v_2+A v_1\otimes v_1
  \end{array}\right],
\end{align*}
where $A=\frac{[(-1)^{\iota}\xi^{-2j}+(-1)^{\iota+1}\xi^{2j}]\xi^4}{1+\xi^5}$. Since $c(v_1\otimes v_1)=\xi^{2j} v_1\otimes v_1$, it follows by the formulae \eqref{eqSmash} and \eqref{eqSkew-1} that
\begin{align*}
\partial_1(v_1^3)&=(v^1\otimes \text{id}^{\otimes 2})\Delta^{1,2}(v_1^3)=(1+\xi^{2j}+\xi^{4j})v_1^2=0,\\
\partial_2(v_1^3)&=(v^2\otimes \text{id}^{\otimes 2})\Delta^{1,2}(v_1^3)=0.
\end{align*}
Similarly, we obtain that
\begin{align*}
\partial_1(v_1^2v_2)&=v_1v_2+(-1)^{\iota+1}\xi^{-2j}v_2v_1,\ \partial_2(v_1^2v_2)=\xi^{2j}v_1^2;\\
\partial_1(v_1v_2v_1)&=(-1)^{\iota+1}\xi^{-2j}v_1v_2,\quad  \partial_2(v_1v_2v_1)=(-1)^{\iota}\xi^{-2j}v_1^2;\\
\partial_1(v_2v_1^2)&=(-1)^{\iota}v_2v_1,\quad  \partial_2(v_2v_1^2)=v_1^2;\\
\partial_1(v_2^3)&=(1+\xi^{4j})Av_1v_2+(-1)^{\iota+1}\xi^{2j}Av_2v_1,\quad\partial_2(v_2^3)=0;\\
\partial_1(v_1v_2^2)&=A\xi^{2j}v_1^2-\xi^{4j}v_2^2,\quad \partial_2(v_1v_2^2)=(-)^{\iota+1}v_1v_2;\\
\partial_1(v_2^2v_1)&=Av_1^2+\xi^{4j}v_2^2,\quad \partial_2(v_2^2v_1)=-\xi^{2j}v_2v_1;\\
\partial_1(v_2v_1v_2)&=(-1)^{\iota}\xi^{-2j}Av_1^2,\quad\partial_2(v_2v_1v_2)=v_1v_2+(-1)^{\iota}\xi^{2j}v_2v_1.
\end{align*}

It is easy to verify that $\partial_1(r)=0=\partial_2(r)$ for any relation $r$ given in \eqref{eqR1-1}--\eqref{eqR1-3}. Then by \eqref{Def-Nichols-IV}, the quotient $\mathfrak{B}$ of $T(V_{i,j,k,\iota})$ by the relations \eqref{eqR1-1}--\eqref{eqR1-3} projects onto $\BN(V_{i,j,k,\iota})$.
 We claim that $I=\K\{v_1^i(v_2v_1)^jv_2^k,\ i,k\in\I_{0,2},j\in\I_{0,1}\}$ is a left ideal. Indeed, from \eqref{eqR1-1}--\eqref{eqR1-3} and $(v_2v_1)^2=(-1)^{\iota}v_1^2v_2^2$, we have $v_1I,v_2I\subset I$. Hence $I$ linearly  generates $\mathfrak{B}$ since  $1\in I$.

 We claim that $\dim\BN(V_{i,j,k,\iota})\geq 18=|I|$. Indeed, the Dynkin diagram of $X\oplus X_{i,j,k,\iota}$ is \xymatrix@C+25pt{\overset{-1 }{{\circ}}\ar
@{-}[r]^{-\xi^{2j}}
& \overset{\xi^{2j}}{{\circ}}}. Since $j\notin\{0,3\}$, it is of  standard type $B_2$.  By \cite{An13,An15},  $\dim\BN(X\oplus X_{i,j,k,\iota})=36$. The claim follows since $\dim\BN(V_{i,j,k,\iota})\geq\frac{1}{2}\dim\BN(X\oplus X_{i,j,k,\iota})$. Consequently, $\BN\cong\BN(V_{i,j,k,\iota})$.
\end{proof}

\begin{pro}\label{pro-2-24}
 Let $\Lambda^2=\{(i,j,k,\iota)\in\Lambda\mid 3(k+\iota-1)-j\equiv\pm 2\mod 6, 3(k+j)(\iota-1)+ij\equiv \pm 2\mod 6\}$.
   The Nichols algebra $\BN(V_{i,j,k,\iota})$ for $(i,j,k,\iota)\in\Lambda^2$ is generated by $v_1,v_2$, subject to the relations
   \begin{gather}
   v_1^3=0,\quad \xi^{2j}v_1^2v_2+(-1)^k\xi^jv_1v_2v_1+v_2v_1^2=0,\label{eqR2-1}\\
   v_2^2v_1+[(-1)^{\iota}\xi^{2j}+(-1)^k\xi^j]v_2v_1v_2-v_1v_2^2=0,\quad v_2^6=0.\label{eqR2-2}
   \end{gather}
\end{pro}
\begin{proof}
The braiding  of $V_{i,j,k,\iota}$ is given by $ c(\left[\begin{array}{ccc} v_1\\v_2\end{array}\right]\otimes\left[\begin{array}{ccc} v_1~v_2\end{array}\right])=$
\begin{align*}
 \left[\begin{array}{ccc}
\xi^{2j} v_1\otimes v_1& (-1)^k\xi^j v_2\otimes v_1+[\xi^{2j}+(-1)^{k+\iota}\xi^j] v_1\otimes v_2\\
(-1)^{\iota+1}\xi^{2j} v_1\otimes v_2 & (-1)^{k+\iota}\xi^j v_2\otimes v_2-(-1)^{\iota}\frac{\xi}{1-\xi^2} v_1\otimes v_1
  \end{array}\right].
  \end{align*}

   Then a direct computation shows that the relations \eqref{eqR2-1} and \eqref{eqR2-2} are zero in $\BN(V_{i,j,k,\iota})$ being annihilated by $\partial_1,\partial_2$ and hence the quotient $\mathfrak{B}$ of $T(V_{i,j,k,\iota})$ by the relations \eqref{eqR2-1} and \eqref{eqR2-2} projects onto $\BN(V_{i,j,k,\iota})$. We claim that $I=\K\{v_1^i(v_2v_1)^jv_2^k,i\in\I_{0,2},j\in\I_{0,1},k\in\I_{0,5}\}$ is a left ideal. Indeed, from \eqref{eqR2-1}, \eqref{eqR2-2} and $(v_2v_1)^2=-[(-1)^{\iota}\xi^{2j}+(-1)^k\xi^j]v_1^2v_2^2-2(v_1v_2)^2$, it is easy to show that $v_1I,v_2I\subset I$. Hence $I$ linearly  generates $\mathfrak{B}$ since  $1\in I$.

 We claim that $\dim\BN(V_{i,j,k,\iota})\geq 36=|I|$. Indeed,
  the Dynkin diagram of $X\oplus X_{i,j,k,\iota}$ is \xymatrix@C+25pt{\overset{-1 }{{\circ}}\ar
@{-}[r]^{\xi^{2j}}
& \overset{\xi^{2j}}{{\circ}}}. Since $j\notin\{0,3\}$, it is of standard type $B_2$ and by \cite{An13,An15},  $\dim\BN(X\oplus X_{i,j,k,\iota})=72$. The claim follows since $\dim\BN(V_{i,j,k,\iota})\geq\frac{1}{2}\dim\BN(X\oplus X_{i,j,k,\iota})$. Consequently, $\BN\cong\BN(V_{i,j,k,\iota})$.
\end{proof}

\begin{pro}\label{pro-3-24}
Let $\Lambda^3=\{(i,j,k,\iota)\in\Lambda\mid 3(k+\iota-1)-j\equiv\mp 2\mod 6, 3(k+j)(\iota-1)+ij\equiv \pm 1\mod 6\}$. The Nichols algebra $\BN(V_{i,j,k,\iota})$ for $(i,j,k,\iota)\in\Lambda^3$ is generated by  $v_1,v_2$, subject to the relations
   \begin{gather}
\notag   v_1^6=0,\quad(-1)^{k+\iota}\xi^{-j}v_1^2v_2+[(-1)^{\iota}+(-1)^k\xi^{-j}]v_1v_2v_1+v_2v_1^2=0,\\
  (-)^{\iota+1}\frac{1}{3}(\xi+\xi^2)v_1^3+v_1v_2^2+(-1)^{\iota}v_2v_1v_2+v_2^2v_1=0,\label{eqR3}\\
\notag  \frac{(-1)^{\iota}\xi^{1-2j}}{1+\xi^5}v_1^2v_2+\frac{(-1)^{k+\iota}\xi^{1-j}}{1+\xi^5}v_1v_2v_1+v_2^3=0.
  \end{gather}
\end{pro}
\begin{proof}
The braiding of $V_{i,j,k,\iota}$ is given by $ c(\left[\begin{array}{ccc} v_1\\v_2\end{array}\right]\otimes\left[\begin{array}{ccc} v_1~v_2\end{array}\right])=$
\begin{align*}
\left[\begin{array}{ccc}
(-1)^{k+\iota}\xi^{-j} v_1\otimes v_1& (-1)^{\iota}\xi^{-2j} v_2\otimes v_1+[\xi^{-2j}+(-1)^{k+\iota}\xi^{-j}] v_1\otimes v_2\\
(-1)^{k+1}\xi^{-j} v_1\otimes v_2 & \xi^{-2j}v_2\otimes v_2+(-1)^{\iota}\frac{\xi}{1-\xi^2} v_1\otimes v_1
  \end{array}\right].
  \end{align*}

  Then a direct computation shows that the relations \eqref{eqR3} are zero in $\BN(V_{i,j,k,\iota})$ being annihilated by $\partial_1,\partial_2$ and hence the quotient $\mathfrak{B}$ of $T(V_{i,j,k,\iota})$ by the relations \eqref{eqR3} projects onto $\BN(V_{i,j,k,\iota})$. We claim that $I=\K\{v_1^i(v_2v_1)^jv_2^k,i\in\I_{0,5},j\in\I_{0,1},k\in\I_{0,2}\}$ is a left ideal. Indeed, from \eqref{eqR3} and $(v_2v_1)^2=-(\xi^{2j}+\xi^{2j+1})v_1^2v_2^2-2\xi^{2j}(v_1v_2)^2$, it is easy to show that $v_1I,v_2I\subset I$. Hence $I$ linearly  generates $\mathfrak{B}$ since clearly $1\in I$.

 We claim that $\dim\BN(V_{i,j,k,\iota})\geq 36=|I|$. Indeed,
  the generalized Dynkin diagram of $\BN(X\oplus X_{i,j,k,\iota})$ is \xymatrix@C+45pt{\overset{-1 }{{\circ}}\ar
@{-}[r]^{(-1)^{k+\iota+1}\xi^{-j}}
& \overset{(-1)^{k+\iota}\xi^{-j}}{{\circ}}}. Since $j\notin\{0,3\}$, it is of  standard type $B_2$. By \cite{An13,An15}, $\dim\BN(X\oplus X_{i,j,k,\iota})=72$. The claim follows since $\dim\BN(V_{i,j,k,\iota})\geq\frac{1}{2}\dim\BN(X\oplus X_{i,j,k,\iota})$. Consequently, $\BN\cong\BN(V_{i,j,k,\iota})$.
\end{proof}
%
%
\begin{pro}\label{pro-456-24}
\begin{enumerate}
\item Let $(i,j,k,\iota)\in\Lambda^4=\{(i,3,0,1),(i,3,1,0),i\in\{1,3,5\}\}$. The Nichols algebra $\BN(V_{i,j,k,\iota})$  is generated as an algebra by $v_1,v_2$, subject to the relations
\begin{align}
v_1^2=0,\quad v_1v_2+(-1)^{\iota}v_2v_1=0,\quad v_2^2=0.\label{eqR4}
\end{align}
\item Let $\Lambda^5=\{(i,j,k,\iota)\in\Lambda-\Lambda^4\mid 3(k+j)(\iota-1)+ij\equiv 3\mod 6\}$.
The Nichols algebra $\BN(V_{i,j,k,\iota})$ for $(i,j,k,\iota)\in\Lambda^5$ is generated by $v_1, v_2$, subject to the relations
\begin{align*}
v_1^2=0,\  v_1v_2+(-1)^k\xi^{5j}v_2v_1=0,\  v_2^N=0,\  N=\ord((-1)^{k+\iota-1}\xi^{-j})\in\{3,6\}.
\end{align*}

\item Let $\Lambda^6=\{(i,j,k,\iota)\in\Lambda-\Lambda^4\mid 3(k+\iota)+3(k+j)(\iota-1)+(i-1)j\equiv 3\mod 6 \}$. The Nichols algebra $\BN(V_{i,j,k,\iota})$ for $(i,j,k,\iota)\in\Lambda^6$ is generated by $v_1,v_2$, subject to the relations
    \begin{gather*}
    v_1v_2+(-1)^{\iota}v_2v_1=0,\; v_2^2+(1-\xi^2)^{-1}\xi^{2+4i}(-1)^{\iota}v_1^2=0,\; v_1^N=0,
    \end{gather*}
    where $ N=\ord((-1)^{\iota+1+k}\xi^j)\in\{3,6\}.$
\end{enumerate}
\end{pro}
\begin{proof}
Assume that $(i,j,k,\iota)\in\Lambda^4$. The braiding of $V_{i,j,k,\iota}=\{v_1,v_2\}$ is given by
  \begin{align*}
  c(\left[\begin{array}{ccc} v_1\\v_2\end{array}\right]\otimes\left[\begin{array}{ccc} v_1~v_2\end{array}\right])=
  \left[\begin{array}{ccc}
- v_1\otimes v_1& (-1)^{\iota-1}v_2\otimes v_1-2v_1\otimes v_2\\
(-1)^{\iota} v_1\otimes v_2 &-v_2\otimes v_2
  \end{array}\right].
  \end{align*}
Then a direct computation shows that $v_1^2,v_2^2,v_1v_2+(-1)^{\iota}v_2v_1\in\Pp(T(V_{i,j,k,\iota}))$. Indeed, we have
  \begin{gather*}
  \Delta(v_i^2)=v_i^2\otimes 1+v_i\otimes v_i+c(v_i\otimes v_i)+1\otimes v_i^2=v_i^2\otimes 1+1\otimes v_i^2;\quad i\in\I_{1,2},\\
  \Delta(v_1v_2+(-1)^{\iota}v_2v_1)=(v_1v_2+(-1)^{\iota}v_2v_1)\otimes 1+1\otimes (v_1v_2+(-1)^{\iota}v_2v_1).
  \end{gather*}
Therefore, the quotient $\mathfrak{B}$ of $T(V_{i,j,k,\iota})$ by the relations \eqref{eqR4} projects onto $\BN(V_{i,j,k,\iota})$.  From \eqref{eqR3}, it is easy to show that  $I=\K\{v_1^iv_2^j,i,j\in\I_{0,1}\}$ is a left ideal and linearly  generates $\mathfrak{B}$.

 We claim that $\dim\BN(V_{i,j,k,\iota})\geq 4=|I|$. Indeed,
  the diagram is \xymatrix@C+15pt{\overset{-1 }{\underset{x  }{\circ}}\ar
@{-}[r]^{-1}
& \overset{-1}{\underset{v_1 }{\circ}}}. It is of Cartan type $A_2$.  By \cite{An13,An15}, $\dim\BN(X\oplus X_{i,j,k,\iota})=8$. The claim follows since $\dim\BN(V_{i,j,k,\iota})=\frac{1}{2}\dim\BN(X\oplus X_{i,j,k,\iota})$. Consequently, $\BN\cong\BN(V_{i,j,k,\iota})$.

The proof follows for $(i,j,k,\iota)\in\Lambda^5$ or $\Lambda^6$ the same lines as for $(i,j,k,\iota)\in\Lambda^4$. In these cases, the generalized Dynkin diagram of $X\oplus X_{i,j,k,\iota}$ is given by
$$
\xymatrix@C+35pt{\overset{-1 }{\underset{x  }{\circ}}\ar
@{-}[r]^{(-1)^{k+\iota-1}\xi^{-j}}
& \overset{-1}{\underset{v_1 }{\circ}}} or \xymatrix@C+45pt{\overset{-1 }{\underset{x  }{\circ}}\ar@{-}[r]^{(-1)^{\iota+1+k}\xi^{-j}}& \overset{(-1)^{\iota+1+k}\xi^{j}}{\underset{v_1 }{\circ}}}.
$$
They are of standard $A_2$ type.
\end{proof}
\begin{thm}\label{thmFDNichols-24}
Let $V$ be a simple object in ${}_{\cK_{24,1}}^{\cK_{24,1}}\mathcal{YD}$ such that $\dim\BN(V)<\infty$. Then $V$ is isomorphic either to $\K_{\chi_{i,j,k}}$ for $(i,j,k)\in\Lambda^0$ or to $V_{i,j,k,\iota}$ for $(i,j,k,\iota)\in\cup_{i=1}^6\Lambda^i$.
\end{thm}
\begin{proof}
By Theorem \ref{thmsimplemoduleD-24}, $V$ is isomorphic to $\K_{\chi_{i,j,k}}$ for $(i,j,k)\in \I_{0,1}\times \I_{0,1}\times \I_{0,5}$ or $V_{i,j,k,\iota}$ for $(i,j,k,\iota)\in\Lambda$. If $\dim V=1$, by Proposition \ref{proNicholsoneH-24}, $V\cong \K_{\chi_{i,j,k}}$ for $(i,j,k)\in\Lambda^0$. Observe that $\Lambda=\Lambda^{0\ast}\cup\Lambda^{0\ast\ast}\cup\cup_{i=1}^6\Lambda^i$. If $\dim V=2$, then by Propositions \ref{proV1-24}---\ref{pro-456-24}, $V\cong V_{i,j,k,\iota}$ for $(i,j,k,\iota)\in\cup_{i=1}^6\Lambda^i$.
\end{proof}
\begin{rmk}
\begin{enumerate}
\item $|\Lambda^1|=12$, $|\Lambda^2|=8=|\Lambda^3|$, $|\Lambda^4|=6$, $|\Lambda^5|=12=|\Lambda^6|$.
\item The Nichols algebras $\BN(V_{i,j,k,\iota})$ with $(i,j,k,\iota)\in\Lambda^4\cup\Lambda^5\cup\Lambda^6$
    have already appeared in \cite{AGi17}. They are isomorphic to quantum planes as algebras. They can be recovered, up to isomorphism,  by using the techniques in \cite{AA18}. Indeed, from the proofs of Proposition \ref{pro-456-24}, they are arising from Nichols algebras of Cartan type $A_2$ or standard type $A_2$.

\item  The Nichols algebra $\BN(V_{i,j,k,\iota})$ for $(i,j,k,\iota)\in\cup_{i=1}^3\Lambda^i$ is an algebra of dimension $18$ or $36$ with no quadratic relations.  They appeared in \cite{X17}(see also \cite{HX18b,MBGG}) with different parameters. Indeed, from the proofs of Propositions \ref{proV1-24}---\ref{pro-456-24}, \cite[Remark 4.21]{X17} and \cite[Proposition 5.9]{MBGG}, up to isomorphism, they are arising from Nichols algebras of standard type $B_2$ whose Dynkin diagram is \xymatrix@C+15pt{\overset{-1 }{{\circ}}\ar
@{-}[r]^{q^{-2}} & \overset{q}{{\circ}}} $( q\in\K^{\times}-\{1,-1\}$ and $q^4\neq 1)$, with different matrices of the braiding.
\end{enumerate}
\end{rmk}

\section{Hopf algebras over $\cK_{24,1}$}\label{secK-lifting}
We determine all finite-dimensional Hopf algebras over $\cK_{24,1}$, whose infinitesimal braidings are simple objects in ${}_{\cK_{24,1}}^{\cK_{24,1}}\mathcal{YD}$. We first define four families of Hopf algebras $\mathcal{C}_{i,j,k,\iota}(\mu)$ for $(i,j,k,\iota)\in\Lambda^{1\ast}$ and show that they are indeed the liftings of $\BN(V_{i,j,k,\iota})\sharp \cK_{24,1}$. Here $\Lambda^{1\ast}=\{(2,j_1,0,0),(2,j_2,1,0),j_1=2,4,j_2=1,5\}\subset\Lambda^1$.
\begin{defi}
For $j\in\{2,4\}$ and $\mu\in\K$, let $\mathcal{C}_{2,j,0,0}(\mu)$  be the algebra generated by   $a,b,c,d,v_1, v_2$, subject to the relations $\eqref{defiH-1}$ and the following ones:
\begin{gather}
 av_1=\xi^5v_1a,\quad av_2=\xi^4v_2a+v_1c, \quad bv_1=\xi^5v_1b,\quad bv_2=\xi^4v_2b+ v_1d, \label{eq-C2j-1} \\
       cv_1=\xi^2v_1c,\quad cv_2=\xi^{4}v_2c+ v_1a,\quad dv_1=\xi^2v_1d,\quad dv_2=\xi^{4}v_2d+ v_1b,\label{eq-C2j-2}\\
 v_1^3=0, \quad\xi^{2j}v_1^2v_2+\xi^{-2j} v_1v_2v_1+v_2v_1^2=0, \label{eq-C2j-3}\\
\frac{(1-\xi^{2j})\xi^4}{1+\xi^5}v_1^2v_2+\frac{(1-\xi^{4j})\xi^4}{1+\xi^5}v_1v_2v_1+v_2^3=\mu(1-da^{-1}),\\
v_1v_2^2+v_2v_1v_2+v_2^2v_1=-2\mu\xi^4ba^{-1}.
\end{gather}
\end{defi}
$\mathcal{C}_{2,j,0,0}(\mu)$ admits a Hopf algebra structure, where the comultiplication is given by  \eqref{defiH-2-24} and
\begin{align}
\begin{split}\label{eqCoproduct-H-1}
  \Delta(v_1)&=v_1\otimes 1+a^j\otimes v_1-\xi(1+\xi^j)ba^{j-1}\otimes v_2,\\
  \Delta(v_2)&=v_2\otimes 1+da^{j-1}\otimes v_2-(1-\xi^2)^{-1}\xi^{-1}(1-\xi^j)ca^{j-1}\otimes v_1.
\end{split}
\end{align}

\begin{rmk}
It is clear that $\mathcal{C}_{2,j,0,0}(0)\cong\BN(V_{2,j,0,0})\sharp \cK_{24,1}$ and $\mathcal{C}_{2,j,0,0}(\mu)$ with $\mu\neq 0$ is not isomorphic to $\mathcal{C}_{2,j,0,0}(0)$ for $j\in\{2,4\}$.
\end{rmk}

\begin{defi}
For $j\in\{1,5\}$ and $\mu\in\K$, let $\mathcal{C}_{2,j,1,0}(\mu)$  be the algebra generated by  $a,d,c,d,v_1,v_2$, subject to the relations  \eqref{defiH-1}, \eqref{eq-C2j-1}-\eqref{eq-C2j-3}  and the following ones:
\begin{gather}
\frac{(1-\xi^{2j})\xi^4}{1+\xi^5}v_1^2v_2+\frac{(1-\xi^{4j})\xi^4}{1+\xi^5}v_1v_2v_1+v_2^3=\mu(1-a^3),\\
v_1v_2^2+v_2v_1v_2+v_2^2v_1=-2\mu\xi^4ca^2.
\end{gather}
\end{defi}
$\mathcal{C}_{2,j,1,0}(\mu)$ admits a Hopf algebra structure, where the comultiplication is given by \eqref{defiH-2-24} and
\begin{align}
\begin{split}\label{eqCoproduct-H-2}
\Delta(v_1)&=v_1\otimes 1+da^{j-1}\otimes v_1-\xi(1-\xi^j)ca^{j-1}\otimes v_2,\\
\Delta(v_2)&=v_2\otimes 1+a^{j}\otimes v_2-(1-\xi^2)^{-1}\xi^{-1}(1+\xi^j)ba^{j-1}\otimes v_1.
\end{split}
\end{align}

\begin{rmk}
It is clear that $\mathcal{C}_{2,j,1,0}(0)\cong\BN(V_{2,j,1,0})\sharp \cK_{24,1}$ and $\mathcal{C}_{2,j,1,0}(\mu)$ with $\mu\neq 0$ is not isomorphic to $\mathcal{C}_{2,j,1,0}(0)$ for $j\in\{1,5\}$.
\end{rmk}

\begin{lem}\label{lemDimA-24-1}
A linear basis of $\mathcal{C}_{i,j,k,\iota}(\mu)$ for $(i,j,k,\iota)\in\Lambda^{1\ast}$ is given by
\begin{gather*}
 \{v_2^i(v_1v_2)^jv_1^kd^{\mu}c^{\nu}b^la^m,~i,k\in\I_{0,2},~j,\mu,\nu,l,\mu+\nu+l\in\I_{0,1},~m\in\I_{0,5}\}.
\end{gather*}
\end{lem}
\begin{proof}
 We prove the assertion for $\mathcal{C}_{2,j,1,0}(\mu)$, being the proof for $\mathcal{C}_{2,j,0,0}(\mu)$ completely analogous.  We write $v_{12}:=v_1v_2$ for short. By Diamond Lemma, it suffices to show that all overlaps ambiguities are resolvable, that is, the ambiguities can be reduced to the same expression by different substitution rules with the order $v_2<v_1v_2<v_1<d<c<b<a$. Here we verify that the overlapping pair  $(fv_2)v_2^2=f(v_2^3)$ for $f\in\{a,\;b,\;c,\;d\}$ is resolvable:
\begin{align*}
(av_2)v_2^2&=(\xi^4v_2a+v_1c)v_2^2
=\xi^2v_2^2av_2+\xi^4(v_1v_2+v_2v_1)cv_2+v_1^2av_2\\
&=\xi^2v_2^2(\xi^4v_2a+v_1c)+\xi^4(v_1v_2+v_2v_1)(\xi^4v_2c+v_1a)+v_1^2(\xi^4v_2a+v_1c)\\
&=v_2^3a+\xi^2v_2^2v_1c+\xi^2v_1v_2^2c+\xi^4v_1v_2v_1a+\xi^2v_2v_1v_2c+\xi^4v_2v_1^2a\\&\quad+\xi^4v_1^2v_2a+v_1^3c\\
&=v_2^3a+\xi^4v_1v_2v_1a+\xi^4v_2v_1^2a+\xi^4v_1^2v_2a+\xi^2(v_2^2v_1+v_1v_2^2+v_2v_1v_2)c\\
&=v_2^3a+\xi^4v_1v_2v_1a+\xi^4v_2v_1^2a+\xi^4v_1^2v_2a\\
&=\frac{(1-\xi^{2j})\xi}{1+\xi^5}av_1^2v_2+\frac{(1-\xi^{4j})\xi}{1+\xi^5}av_1v_2v_1+\mu a(1-a^3)=a(v_2^3).
\end{align*}
\vspace{-0.4cm}
\begin{align*}
(bv_2)v_2^2&=(\xi^4v_2b+ v_1d)v_2^2
=\xi^2v_2^2bv_2+\xi^4v_2v_1dv_2+\xi^4v_1v_2dv_2+v_1^2bv_2\\
&=v_2^3b+\xi^2v_2^2v_1d+\xi^2v_2v_1v_2d+\xi^4v_2v_1^2b+\xi^2v_1v_2^2d+\xi^4v_1v_2v_1b\\&\quad+\xi^4v_1^2v_2b+v_1^3d\\
&=v_2^3b+\xi^4v_1v_2v_1b+\xi^4v_2v_1^2b+\xi^4v_1^2v_2b+\xi^2(v_2^2v_1+v_1v_2^2+v_2v_1v_2)d\\
&=v_2^3b+\xi^4v_1v_2v_1b+\xi^4v_2v_1^2b+\xi^4v_1^2v_2b-2\mu cda^2\\
&=\frac{(1-\xi^{2j})\xi}{1+\xi^5}bv_1^2v_2+\frac{(1-\xi^{4j})\xi}{1+\xi^5}bv_1v_2v_1+\mu b(1-a^3)=b(v_2^3).
\end{align*}
\vspace{-0.4cm}
\begin{align*}
(dv_2)v_2^2&=(\xi^{4}v_2d+ v_1b)v_2^2
=\xi^2v_2^2dv_2+\xi^4(v_1v_2+v_2v_1)bv_2+v_1^2dv_2\\
&=v_2^3d+\xi^2v_2^2v_1b+\xi^2v_1v_2^2b+\xi^4v_1v_2v_1d+\xi^2v_2v_1v_2b+\xi^4v_2v_1^2d\\&\quad+\xi^4v_1^2v_2d+v_1^3b\\
&=v_2^3d+\xi^4v_1v_2v_1d+\xi^4v_2v_1^2d+\xi^4v_1^2v_2d+\xi^2(v_2^2v_1+v_1v_2^2+v_2v_1v_2)b+v_1^3b\\
&=\frac{(1-\xi^{2j})\xi}{1+\xi^5}dv_1^2v_2+\frac{(1-\xi^{4j})\xi}{1+\xi^5}dv_1v_2v_1+\mu d(1-a^3)=d(v_2^3).\\
(cv_2)v_2^2&=(\xi^4v_2c+ v_1a)v_2^2\\
&=v_2^3c+\xi^4v_2v_1^2c+\xi^4v_1v_2v_1c+\xi^4v_1^2v_2c+\xi^2(v_2^2v_1+v_2v_1v_2+v_1v_2^2)a+v_1^3a\\
&=\frac{(1-\xi^{2j})\xi}{1+\xi^5}cv_1^2v_2+\frac{(1-\xi^{4j})\xi}{1+\xi^5}cv_1v_2v_1+\mu c(1-a^3)=c(v_2^3).
\end{align*}
One can also show that the remaining overlaps
\begin{gather*}
\{(fv_1)v_1^2,  \quad
\{(fv_{12})v_{12},f(v_{12}^2)\}, \quad
\{(v_1v_{12})v_{12}, v_1(v_{12}^2)\},\quad \{(v_1^3)v_{12}, v_1^2(v_1v_{12})\},\\ \{(v_1^3)v_2,v_1^2(v_1v_2)\},\quad
\{(v_1(v_2^3),(v_1v_2)v_2^2\} \quad \{v_{12}(v_2^3), (v_{12}v_2)v_2^2\},\\
\{(v_{12}^2)v_2, v_{12}(v_{12}v_2)\},\quad \{v_i^3v_i^r,v_i^rv_i^3\},\quad \{v_{12}^2v_{12}^h,v_{12}^hv_{12}^2\},
\end{gather*}
are resolvable by using the defining relations.  Here we omit the details since it is tedious but straightforward.
\end{proof}

\begin{lem}\label{lemCC-H-24}
For $(i,j,k,\iota)\in\Lambda^{1\ast}$, $\gr\mathcal{C}_{i,j,k,\iota}(\mu)\cong \BN(V_{i,j,k,\iota})\sharp \cK_{24,1}$.
\end{lem}
\begin{proof}
Let $\mathfrak{C}_{0}$ be the Hopf subalgebra of $\mathcal{C}_{i,j,k,\iota}(\mu)$   generated by the simple subcoalgebra $\K\{a,b,c,d\}$.  By Lemma \ref{lemDimA-24-1}, $\dim \mathfrak{C}_{0}=24$. It is clear that $\mathfrak{C}_{0}\cong \cK_{24,1}$. Let $\mathfrak{C}_{n}=\mathfrak{C}_{n-1}+\cK_{24,1}\{y^i(xy)^jx^k,i+2j+k=n,i,k\in\I_{0,2},j\in\I_{0,1}\}$ for $n\in\I_{0,6}$. A direct computation shows that $\{\mathfrak{C}_{n}\}_{n\in\I_{0,6}}$ is a coalgebra filtration of $\mathcal{C}_{i,j,k,\iota}(\mu)$ and hence the coradical $(\mathcal{C}_{i,j,k,\iota}(\mu))_{0}\subset \mathfrak{C}_{0}\cong \cK_{24,1}$, which implies that $(\mathcal{C}_{i,j,k,\iota}(\mu))_{[0]}\cong \cK_{24,1}$ and $\gr\mathcal{C}_{i,j,k,\iota}(\mu)\cong R_{i,j,k,\iota}\sharp \cK_{24,1}$, where $R_{i,j,k,\iota}$ is a connected Hopf algebra in ${}_{\cK_{24,1}}^{\cK_{24,1}}\mathcal{YD}$. Since  $V_{i,j,k,\iota}\subset\Pp(R_{i,j,k,\iota})$ by  definition and $\dim R_{i,j,k,\iota}=18=\dim\BN(V_{i,j,k,\iota})$ by Lemma \ref{lemDimA-24-1}, it follows that $R_{i,j,k,\iota}\cong\BN(V_{i,j,k,\iota})$  and  consequently, $\gr\mathcal{C}_{i,j,k,\iota}(\mu)\cong \BN(V_{i,j,k,\iota})\sharp \cK_{24,1}$.
\end{proof}
\begin{pro}\label{pro-Alg-1-24}
Let $A$ be a finite-dimensional Hopf algebra over $\cK_{24,1}$ such that $\gr A\cong\BN(V_{i,j,k,\iota})\sharp \cK_{24,1}$ for $(i,j,k,\iota)\in\Lambda^{1\ast}$. Then $A\cong\mathcal{C}_{i,j,k,\iota}(\mu)$ for some $\mu\in\K$.
\end{pro}
\begin{proof}
Let $X=\xi^{2j}v_1^2v_2+\xi^{-2j} v_1v_2v_1+v_2v_1^2$, $Y=\frac{(1-\xi^{2j})\xi^4}{1+\xi^5}v_1^2v_2+\frac{(1-\xi^{4j})\xi^4}{1+\xi^5}v_1v_2v_1+v_2^3$ and $Z=v_1v_2^2+v_2v_1v_2+v_2^2v_1$ for simplicity.
Assume that $(i,j,k,\iota)=(2,j,0,0)$ for some $j\in\{2,4\}$. Note that $\gr A\cong\BN(V_{i,j,k,\iota})\sharp \cK_{24,1}$. By \eqref{eqCoproduct-H-1}, a direct computation shows that
\begin{gather*}
\Delta(v_1^3)=v_1^3\otimes 1+1\otimes v_1^3+(1-\xi^2)^{-\frac{1}{2}}x_2ba^{-1}\otimes X,\\
\Delta(X)=X\otimes 1+da^{-1}\otimes X,\quad \Delta(Y)=Y\otimes 1+da^{-1}\otimes Y,\\
\Delta(Z)=Z\otimes 1+1\otimes Z+(1-\xi^2)^{-\frac{1}{2}}x_1ba^{-1}\otimes X
-2\xi^j(1-\xi^2)^{-\frac{1}{2}}x_2ba^{-1}\otimes Y.
\end{gather*}
It follows that $X,Y\in\Pp_{1,da^{-1}}(A)=\Pp_{1,da^{-1}}(\cK_{24,1})=\K\{1-da^{-1},ca^{-1}\}$, that is, $X=\alpha_1(1-da^{-1})+\alpha_2ca^{-1}, Y=\beta_1(1-da^{-1})+\beta_2ca^{-1}$ for some $\alpha_1,\alpha_2,\beta_1,\beta_2\in\K$. Then
\begin{align*}
&\Delta(v_1^3+\alpha_1(1-\xi^2)^{-\frac{1}{2}}x_2ba^{-1})
=(v_1^3+\alpha_1(1-\xi^2)^{-\frac{1}{2}}x_2ba^{-1})\otimes 1+\\&\quad\quad\quad\quad 1\otimes(v_1^3+\alpha_1(1-\xi^2)^{-\frac{1}{2}}x_2ba^{-1})+(1-\xi^2)^{-\frac{1}{2}}x_2ba^{-1}\otimes\alpha_2ca^{-1}.
\end{align*}
If the relation $v_1^3=0$ admits a non-trivial deformation, then $v_1^3\in A_{[2]}$. Since $av_1^3=-v_1^3a$, $bv_1^3=-v_1^3b$, $cv_1^3=v_1^3c$ and $dv_1^3=v_1^3c$, a tedious computation on $A_{[2]}$ shows that $v_1^3=0$ must hold in $A$. Therefore, the last equation holds only if $\alpha_1=0=\alpha_2$, which implies that $X=0$ in $A$. Similarly, we have that $Y=\beta_1(1-da^{-1})$ and $Z=2\beta_1\xi^j(1-\xi^2)^{-\frac{1}{2}}x_2ba^{-1}=-2\beta_1\xi^4ba^{-1}$. Therefore, the defining relations of $\mathcal{C}_{2,j,0,0}(\beta_1)$ hold in $A$ and hence  there is a surjective Hopf algebra morphism from $\mathcal{C}_{2,j,0,0}(\beta_1)$ to $A$. By Lemma \ref{lemDimA-24-1}, $\dim A=\dim \mathcal{C}_{2,j,0,0}(\beta_1)$ and hence $A\cong \mathcal{C}_{2,j,0,0}(\beta_1)$.

Assume that $(i,j,k,\iota)=(2,j,1,0)$ for some $j\in\{1,5\}$. By \eqref{eqCoproduct-H-2}, a direct computation shows that
\begin{gather*}
\Delta(v_1^3)=v_1^3\otimes 1+da^2\otimes v_1^3+(1-\xi^2)^{-\frac{1}{2}}x_2ca^2\otimes X,\\
\Delta(X)=X\otimes 1+a^3\otimes X,\quad
\Delta(Y)=Y\otimes 1+a^3\otimes Y,\\
\Delta(Z)=Z\otimes 1+da^2\otimes Z+(1-\xi^2)^{-\frac{1}{2}}x_1ca^2\otimes X
+2\xi^j(1-\xi^2)^{-\frac{1}{2}}x_2ca^2\otimes Y.
\end{gather*}
It follows that $X=\alpha_1(1-a^3), Y=\alpha_2(1-a^3)$ for some $\alpha_1,\alpha_2\in\K$. Then $v_1^3+\alpha_1(1-\xi^2)^{-\frac{1}{2}}x_2ca^2\in\Pp_{1,da^2}(A)=\Pp_{1,da^2}(\cK_{24,1})$, that is, $v_1^3+\alpha_1(1-\xi^2)^{-\frac{1}{2}}x_2ca^2=\alpha_3(1-da^2)$ for some $\alpha_3\in\K$. Since $av_1^3=-v_1^3a$, $bv_1^3=-v_1^3b$, $cv_1^3=v_1^3c$ and $dv_1^3=v_1^3c$, it follows that $\alpha_1=0=\alpha_3$ and hence  $v_1^3=0=X$  in $A$. Then $Z+2\alpha_2\xi^j(1-\xi^2)^{-\frac{1}{2}}x_2ca^2\in\Pp_{1,da^2}(A)=\Pp_{1,da^2}(\cK_{24,1})$, that is, $Z+2\alpha_2\xi^j(1-\xi^2)^{-\frac{1}{2}}x_2ca^2=\alpha_4(1-da^2)$ for some $\alpha_4\in\K$.
Since $aZ=\xi Za$, it follows that $\alpha_4=0$ and hence $Z=-2\alpha_2\xi^j(1-\xi^2)^{-\frac{1}{2}}x_2ca^2=-2\alpha_2\xi^4ca^2$. Therefore, there is a surjective Hopf algebra morphism from $\mathcal{C}_{2,j,1,0}(\alpha_2)$ to $A$.  Since $\dim A=\dim \mathcal{C}_{2,j,1,0}(\alpha_2)$ by Lemma \ref{lemDimA-24-1}, $A\cong \mathcal{C}_{2,j,1,0}(\alpha_2)$.
\end{proof}

\begin{pro}\label{pro-Alg-04-24}
Let $A$ be a finite-dimensional Hopf algebra over $\cK_{24,1}$ such that $\gr A\cong\BN(V)\sharp \cK_{24,1}$, where $V$ is isomorphic either to $\K_{\chi_{i,j,k}}$ for $(i,j,k)\in\Lambda^0$ or to $V_{i,j,k,\iota}$ for $(i,j,k,\iota)\in\Lambda^4$. Then $A\cong \gr A$.
\end{pro}
\begin{proof}
Assume that $V\cong\K_{\chi_{i,j,k}}$ for $(i,j,k)\in\Lambda^0$. Since $\Delta(v)=v\otimes 1+d^{j}a^{3i-j}\otimes v$, we have $\Delta(v^2)=v^2\otimes 1+1\otimes v^2$ and hence $v^2=0$ in $A$. Consequently, $A\cong\gr A$.

Assume that $V\cong V_{i,j,k,\iota}$ for $(i,j,k,\iota)\in\Lambda^4$. A direct computation shows that
\begin{gather}
\Delta(v_1^2)=v_1^2\otimes 1+1\otimes v_1^2+(-1)^{\iota}(1-\xi^2)^{-\frac{1}{2}}x_2ba^{-1}\otimes (v_1v_2+(-1)^{\iota}v_1v_2),\label{eqLH-4-1}\\
\Delta(v_2^2)=v_2^2\otimes 1+1\otimes v_2^2,\label{eqLH-4-2}\\
\Delta(v_1v_2+(-1)^{\iota}v_1v_2)=(v_1v_2+(-1)^{\iota}v_1v_2)+da^{-1}\otimes(v_1v_2+(-1)^{\iota}v_1v_2).\label{eqLH-4-last}
\end{gather}
It follows by \eqref{eqLH-4-2} that $v_2^2=0$  in $A$ and by \eqref{eqLH-4-last} that $v_1v_2+(-1)^{\iota}v_2v_1\in\Pp_{1,da^{-1}}(A)=\Pp_{1,da^{-1}}(\cK_{24,1})$, that is, $v_1v_2+(-1)^{\iota}v_2v_1=\alpha_1(1-da^{-1})+\alpha_2ca^{-1}$ for some $\alpha_1,\alpha_2\in\K$. Let $X:=v_1^2+(-1)^{\iota}\alpha_1(1-\xi^2)^{-\frac{1}{2}}x_2ba^{-1}$ for short. Then \eqref{eqLH-4-1} can be rewritten by
\begin{align}
\Delta(X)=X \otimes 1+1\otimes X+(-1)^{\iota}(1-\xi^2)^{-\frac{1}{2}}x_2ba^{-1}\otimes \alpha_2ca^{-1}.\label{eqLH-4-3}
\end{align}
If the relation $v_1^2=0$ admits non-trivial deformations, then $v_1^2\in A_{[1]}$. Since $av_1^2=\xi^{2i}v_1^2a$, $bv_1^2=\xi^{2i}v_1^2b$, $dv_1^2=\xi^{2i}v_1^2d$ and $cv_1^2=\xi^{2i}v_1^2c$, a direct computation on $A_{[1]}$ shows that $v_1^2=0$ in $A$ and hence \eqref{eqLH-4-3} holds  only if $\alpha_1=0=\alpha_2$, which implies that $v_1v_2+(-1)^{\iota}v_2v_1=0$ in $A$.
\end{proof}

\begin{pro}\label{pro-Alg-123-24}
Let $A$ be a finite-dimensional Hopf algebra over $\cK_{24,1}$ such that $\gr A\cong\BN(V_{i,j,k,\iota})\sharp \cK_{24,1}$ for $(i,j,k,\iota)\in\Lambda^1\cup\Lambda^2\cup\Lambda^3-\Lambda^{1\ast}$. Then $A\cong\gr A$.
\end{pro}
\begin{proof}
Assume that $(i,j,k,\iota)\in\Lambda^1-\Lambda^{1\ast}$. Let $X=\xi^{2j}v_1^2v_2+(-1)^{\iota}\xi^{-2j} v_1v_2v_1+v_2v_1^2$, $Y=\frac{(-1)^{\iota}(1-\xi^{2j})\xi^4}{1+\xi^5}v_1^2v_2+\frac{(1-\xi^{4j})\xi^4}{1+\xi^5}v_1v_2v_1+v_2^3$ and $Z=v_1v_2^2+(-1)^{\iota}v_2v_1v_2+v_2^2v_1$ for simplicity. If $k=0$, then we have
\begin{gather}
\Delta(v_1^3)=v_1^3\otimes 1+a^{3j}\otimes v_1^3+(1-\xi^2)^{-\frac{1}{2}}x_2ba^{3j-1}\otimes X,\label{eqV1-1}\\
\Delta(X)=X\otimes 1+da^{3j-1}\otimes X,\quad \Delta(Y)=Y\otimes 1+da^{3j-1}\otimes Y,\label{eqV1-2}\\
\Delta(Z)=Z\otimes 1+a^{3j}\otimes Z-(\xi^{-j}+\xi^j)(1-\xi^2)^{-\frac{1}{2}}x_1ba^{3j-1}\otimes X+\\
\notag 2(-1)^{j+1}\xi^j(1-\xi^2)^{-\frac{1}{2}}x_2ba^{3j-1}\otimes Y.\label{eqV1-3}
\end{gather}

If $j\in\{2,4\}$, then $i-5=k=\iota=0$ and by \eqref{eqV1-2}, $X,Y\in\Pp_{1,da^{-1}}(A)=\Pp_{1,da^{-1}}(\cK_{24,1})=\K\{1-da^{-1},ca^{-1}\}$, that is, $X=\alpha_1(1-da^{-1})+\alpha_2ca^{-1}, Y=\beta_1(1-da^{-1})+\beta_2ca^{-1}$ for some $\alpha_1,\alpha_2,\beta_1,\beta_2\in\K$. Set $r:=v_1^3+\alpha_1(1-\xi^2)^{-\frac{1}{2}}x_2ba^{-1}$. Then
\begin{align}\label{eqV1-4}
\Delta(r)&=r\otimes 1+1\otimes r+(1-\xi^2)^{-\frac{1}{2}}x_2ba^{-1}\otimes\alpha_2ca^{-1}.
\end{align}
If the relation $v_1^3=0$ admits non-trivial deformations, then $v_1^3\in A_{[2]}$. Since $av_1^3=v_1^3a$, $bv_1^3=v_1^3b$, $cv_1^3=-v_1^3c$ and $dv_1^3=-v_1^3d$, a tedious computation on $A_{[2]}$ shows that $v_1^3=0$ must hold in $A$, which implies that \eqref{eqV1-4} holds only if $\alpha_1=0=\alpha_2$ and hence $X=0$ in $A$.
 Similarly, $Y=\beta_1(1-da^{-1})$ and $Z=-2\beta_1\xi^j(1-\xi^2)^{-\frac{1}{2}}x_2ba^{-1}$.
Since $aY=-Ya$ and $ab=\xi ba$, it follows that $\beta_1=0$. Consequently, $A\cong\gr A$.

If $j\in\{1,5\}$, then $i-2=k=\iota-1=0$ and by \eqref{eqV1-2}, $X,Y\in\Pp_{1,da^2}(A)=\Pp_{1,da^2}(\cK_{24,1})=\K\{1-da^2\}$, that is, $X=\alpha_1(1-da^2)$ and $Y=\alpha_2(1-da^2)$ for some $\alpha_1,\alpha_2\in\K$. Moreover, $v_1^3+\alpha_1(1-\xi^2)^{-\frac{1}{2}}x_2ba^2\in\Pp_{1,a^3}(A)=\Pp_{1,a^3}(\cK_{24,1})$, which implies that
$v_1^3+\alpha_1(1-\xi^2)^{-\frac{1}{2}}x_2ba^2=\alpha_3(1-a^3)$. Since $dv_1^3=v_1^3d$ and $cv_1^3=v_1^3c$, it follows that
$\alpha_1=0=\alpha_3$ and hence $X=0=v_1^3$  in $A$. Then a direct computation shows that $Z+2\alpha_2\xi^j(1-\xi^2)^{-\frac{1}{2}}x_2ba^2\in\Pp_{1,a^3}(A)=\Pp_{1,a^3}(\cK_{24,1})$, that is,
$Z+2\alpha_2\xi^j(1-\xi^2)^{-\frac{1}{2}}x_2ba^2=\alpha_4(1-a^3)$ for some $\alpha_4\in\K$.
Since $aZ=\xi^4Za$, $bZ=\xi^4Zb$, $cZ=\xi^4Zc$ and $dZ=\xi^4Zd$, it follows that $\alpha_2=0=\alpha_4$. Consequently, $\gr A\cong A$.

If $k=1$, then $(i,j,k,\iota)\in\{(2,j_1,1,1),(5,j_1,1,1)\mid j_1=2,4\}$ and we have that
\begin{gather*}
\Delta(v_1^3)=v_1^3\otimes 1+da^{3j-1}\otimes v_1^3+(1-\xi^2)^{-\frac{1}{2}}x_2ca^{3j-1}\otimes X,\\
\Delta(X)=X\otimes 1+a^{3j}\otimes X,\quad
\Delta(Y)=Y\otimes 1+a^{3j}\otimes Y,\\
\Delta(Z)=Z\otimes 1+da^{3j-1}\otimes Z+(\xi^{-j}+\xi^j)(1-\xi^2)^{-\frac{1}{2}}x_1ca^{3j-1}\otimes X
+\\2(-1)^{j+1}\xi^j(1-\xi^2)^{-\frac{1}{2}}x_2ca^{3j-1}\otimes Y.
\end{gather*}

It follows that $X=0=Y$ in $A$ and $v_1^3,Z\in\Pp_{1,da^{-1}}(A)=\Pp_{1,da^{-1}}(\cK_{24,1})$, that is,
$v_1^3=\alpha_1(1-da^{-1})+\alpha_2ca^{-1}$ and $Z=\beta_1(1-da^{-1})+\beta_2ca^{-1}$ for some $\alpha_1,\alpha_2,\beta_1,\beta_2\in\K$. Since $av_1^3=(-1)^{i}v_1^3a$, $bv_1^3=(-1)^{i}v_1^3b$, $cv_1^3=(-1)^{i}v_1^3c$ and $dv_1^3=(-1)^{i}v_1^3d$, it follows that $\alpha_1=0=\alpha_2$.
Since $aZ=(-1)^{i}\xi^4Za$, $bZ=(-1)^{i}\xi^4Zb$, $cZ=(-1)^{i}\xi^4Zc$ and $dZ=(-1)^{i}\xi^4Zd$, it follows that $\beta_1=0=\beta_2$. Consequently, $A\cong\gr A$.

The proof for $(i,j,k,\iota)\in\Lambda^2\cup\Lambda^3$ follows the same lines as for $(i,j,k,\iota)\in\Lambda^1$.
\end{proof}

\begin{pro}\label{pro-alge-56-24}
Let $A$ be a finite-dimensional Hopf algebra over $\cK_{24,1}$ such that $\gr A\cong\BN(V_{i,j,k,\iota})\sharp \cK_{24,1}$ for $(i,j,k,\iota)\in\Lambda^5\cup\Lambda^6$. Then $A\cong \gr A$.
\end{pro}
\begin{proof}
Assume that $(i,j,k,\iota)\in\Lambda^5$. Observe that $i\in\{0,3\}$. If the relation $v_1^2=0$ admits non-trivial deformations, then $v_1^2\in A_{[1]}$. Hence there exist some elements $\alpha_{p,q,r},\beta_{p,q,r},\gamma_{p,q,r},\lambda_{p,q,r}\in\K$ with $p+q,p,q\in\I_{0,1},r\in\I_{0,5}$ such that
\begin{align*}
v_1^2=\sum \alpha_{p,q,r}v_1^pv_2^qa^r+\beta_{p,q,r}v_1^pv_2^qda^r+\gamma_{p,q,r}v_1^pv_2^qba^r+\lambda_{p,q,r}v_1^pv_2^qca^r.
\end{align*}
Since $av_1^2=v_1^2a, bv_1^2=v_1^2b, cv_1^2=v_1^2c,dv_1^2=v_1^2d$, it follows that $\alpha_{p,q,r}=\beta_{p,q,r}=\gamma_{p,q,r}=\lambda_{p,q,r}=0$ with $p+q,p,q\in\I_{0,1},r\in\I_{0,5}$ and hence $v_1^2=0$ in $A$. Similarly, we have that $v_1v_2+\xi^{5j}v_2v_1=0$ since
\begin{align*}
a(v_1v_2+(-1)^k\xi^{5j}v_2v_1)&=\xi^{-1}(v_1v_2+(-1)^k\xi^{5j}v_2v_1)a, \\ b(v_1v_2+(-1)^k\xi^{5j}v_2v_1)&=\xi^{-1}(v_1v_2+(-1)^k\xi^{5j}v_2v_1)b, \\ c(v_1v_2+(-1)^k\xi^{5j}v_2v_1)&=-\xi^{-1}(v_1v_2+(-1)^k\xi^{5j}v_2v_1)c,\\ d(v_1v_2+(-1)^k\xi^{5j}v_2v_1)&=-\xi^{-1}(v_1v_2+(-1)^k\xi^{5j}v_2v_1)d.
\end{align*}
If $N=6$, then the relation $v_2^6=0$ must hold in $A$ since $\Delta(v_2^6)=v_2^6\otimes 1+1\otimes v_2^6$. If $N=3$, then the relation $v_2^3=0$ must hold in $A$ since $av_2^3=(-1)^{i+\iota}v_2^3a, bv_2^3=(-1)^{i+\iota}v_2^3b, cv_2^3=(-1)^{i}v_2^3c, dv_2^3=(-1)^{i}v_2^3d$ and $v_2^3\in\Pp_{1,da^{3j-1}}(A)$ or $\Pp_{1,a^3}(A)$. Consequently, $A\cong\gr A$.

The proof for $(i,j,k,\iota)\in\Lambda^6$ follows the same lines as for $(i,j,k,\iota)\in\Lambda^5$.
\end{proof}

Finally, we give the classification of finite-dimensional Hopf algebras over $\cK_{24,1}$ whose infinitesimal braidings are indecomposable objects in ${}_{\cK_{24,1}}^{\cK_{24,1}}\mathcal{YD}$.
\begin{thm}\label{thmFDHopfalgebra-24}
Let $A$ be a finite-dimensional Hopf algebra over $\cK_{24,1}$ whose infinitesimal braiding $V$ is indecomposable
 in ${}_{\cK_{24,1}}^{\cK_{24,1}}\mathcal{YD}$.   Then $A$ is isomorphic  to one of the following objects:
\begin{itemize}
  \item $\bigwedge\K_{\chi_{i,j,k}}\sharp \cK_{24,1}$ for $(i,j,k)\in\Lambda^0$;
  \item $\BN(V_{i,j,k,\iota})\sharp \cK_{24,1}$ for $(i,j,k,\iota)\in\cup_{i=1}^6\Lambda^i-\Lambda^{1\ast}$;
  \item $\mathcal{C}_{i,j,k,\iota}(\mu)$ for $\mu\in\K$ and $(i,j,k,\iota)\in\Lambda^{1\ast}$.
\end{itemize}
\end{thm}
\begin{proof}
 Since $A_{[0]}\cong \cK_{24,1}$,  $\gr A\cong R\sharp \cK_{24,1}$. By \cite[Theorem 1.3]{AA18}, $V$ is simple and $R\cong \BN(V)$. Then by Theorem \ref{thmFDNichols-24}, $V$ is isomorphic either to $\K_{\chi_{i,j,k}}$ for $(i,j,k)\in\Lambda^0$
or to $V_{i,j,k,\iota}$ for $(i,j,k,\iota)\in\cup_{i=1}^6\Lambda^i$.  The theorem follows by Propositions \ref{pro-Alg-1-24}--\ref{pro-alge-56-24}. The Hopf algebras from different families are pairwise non-isomorphic since the diagrams are not isomorphic as Yetter-Drinfeld modules over $\cK_{24,1}$.
\end{proof}
\begin{rmk}
\begin{itemize}
\item    $\bigwedge\K_{\chi_{i,j,k}}\sharp \cK_{24,1}$ with $(i,j,k)\in\Lambda^0$ are basic Hopf algebras of dimension $48$.
\item For $(i,j,k,\iota)\in \Lambda^4$, $\Lambda^5$~and~$\Lambda^6$,
$\BN(V_{i,j,k,\iota})\sharp \cK_{24,1}$ are basic Hopf algebras of dimension $96$, $144$~and~$288$, respectively.
\item For $(i,j,k,\iota)\in\Lambda^1$~and~$\Lambda^2\cup\Lambda^3$, $\BN(V_{i,j,k,\iota})\sharp \cK_{24,1}$ are basic Hopf algebras of dimension~$432$~and~$864$, respectively.
\item $\mathcal{C}_{i,j,k,\iota}(\mu)$ with $\mu\neq 0$ are non-trivial liftings of $\BN(V_{i,j,k,\iota})\sharp \cK_{24,1}$ for $(i,j,k,\iota)\in\Lambda^{1\ast}\subset\Lambda^1$. They constitute new examples of Hopf algebras without the dual Chevally property.
\end{itemize}
\end{rmk}

\vskip10pt \centerline{\bf ACKNOWLEDGMENT}

\vskip10pt

The main part of the paper was written during the visit of the author to University of Padova supported by China Scholarship Council. The author would like to thank his supervisors Profs. Naihong Hu and Giovanna Carnovale for their kind help and continued encouragement. The author is grateful to Profs. N. Andruskiewitsch and G. A. Garcia for helpful comments and guiding discussions on this topic. The author would like to thank  the referee for careful reading and helpful suggestions that largely improved the exposition. The author thanks the referee for careful reading and helpful comments.

\end{document}